\documentclass[preprint]{imsart}

\RequirePackage[numbers]{natbib}
\RequirePackage[colorlinks,citecolor=blue,urlcolor=blue]{hyperref}
\usepackage[american]{babel}
\usepackage{amsmath,amssymb,amsthm,amsfonts,graphicx}
\usepackage{enumitem}
\usepackage[T1]{fontenc}
\usepackage{lipsum}
\usepackage{morefloats}
\usepackage{color}
\usepackage{natbib}
\usepackage{verbatim}
\usepackage{algorithm}
\usepackage{algpseudocode}

\newcommand{\R}{\mathbb{R}}
\newcommand{\E}{\mathbb{E}}
\newcommand{\change}[1]{\textcolor{black}{#1}}

\doi{10.1214/154957804100000000}
\pubyear{0000}
\volume{0}
\firstpage{1}
\lastpage{1}

\startlocaldefs
\newtheorem{theorem}{Theorem}[section]
\newtheorem{corollary}[theorem]{Corollary}
\newtheorem{remark}[theorem]{Remark}
\newtheorem{definition}[theorem]{Definition}
\endlocaldefs

\begin{document}

\begin{frontmatter}

\title{On the choice of the splitting ratio for the split likelihood ratio test}
\runtitle{Asymptotic theory of the split likelihood ratio test}


\author{\fnms{David} \snm{Strieder}\corref{}\ead[label=e1]{david.strieder@tum.de}}
\and
\author{\fnms{Mathias} \snm{Drton}\ead[label=e2]{mathias.drton@tum.de}}
\address{Technical University of Munich; TUM School of Computation, Information and Technology, \\
	Munich Center for Machine Learning (MCML), Munich Data Science Institute (MDSI)  \\
	\printead{e1,e2}}

\runauthor{D. Strieder and M. Drton}

\begin{abstract}
The recently introduced framework of universal inference provides a new approach to constructing hypothesis tests and confidence regions that are valid in finite samples and do not rely on any specific regularity assumptions on the underlying statistical model.  At the core of the methodology is a split likelihood ratio statistic, which is formed under data splitting and compared to a cleverly selected universal critical value.  As this critical value can be very conservative, it is interesting to mitigate the potential loss of power by careful choice of the ratio according to which data are split.  Motivated by this problem, we study the split likelihood ratio test under local alternatives and introduce the resulting class of noncentral split chi-square distributions.  We investigate the properties of this new class of distributions and use it to numerically examine and propose an optimal choice of the data splitting ratio for tests of composite hypotheses of different dimensions.
\end{abstract}




\tableofcontents

\end{frontmatter}

\section{Introduction}

Likelihood ratio tests provide powerful solutions to a broad range of hypothesis testing problems.  However, their implementation generally relies on asymptotic approximations whose validity requires the underlying statistical models to satisfy a number of regularity conditions.  When these conditions are not met, the needed distribution-theoretic insights may be difficult to obtain; see, e.g., \cite{drton:2009,hartigan:1985}.  
Recent work of \citet{universal} addresses this challenge by providing a \textit{split likelihood ratio test} that is universally applicable to problems with i.i.d.~samples.  In this \textit{universal inference} methodology, the data are split into two parts:  one part is used to form a maximum likelihood estimate of a distribution under the full model, and the remaining data are used to compare \change{the} likelihood under the estimate versus the null hypothesis.  Crucially, the independence of the split data allows one to apply a universal critical value, which merely depends on the chosen significance level and is guaranteed to be conservative even for finite samples.  The resulting methodology makes it possible to conduct rather simple analyses of complicated composite hypotheses.  For example, \citet{uai} recently used the approach to construct hypothesis tests for causal effects in a setting with unknown causal structure. 
	
The initial work in \cite{universal} and the follow-up paper by \citet{universal2} investigate the performance of the universal inference framework in the Gaussian case and under consideration of point hypotheses/construction of confidence regions.  Unsurprisingly, the universal framework is rather conservative. To cite the authors: ``our methods may not be optimal, though we do not yet fully understand how close to optimal they are beyond special cases (uniform, Gaussian).''  \change{During the review period, the independent work of \citet{tsedavison} was announced to appear, which covers different aspects of related problems.}

The goal of the present paper is to expand our insights about the behavior of the split likelihood ratio test, as introduced in more detail in Section~\ref{background}.  In particular, we seek to shed light on the impact of the dimensionality of the tested null and alternative hypotheses.  To this end, we study the case of smooth hypotheses \change{ in regular parametric models that are differentiable in quadratic mean. Under similar conditions, \citet{universal} studied in their initial work the diameter of confidence sets corresponding to the inverted split likelihood ratio tests. We extend our insights about the limit behaviour of the split likelihood ratio test by deriving precise limit theory and calculating} the large-sample asymptotic distribution allowing for local alternatives.  This distribution belongs to a ``split-version'' of noncentral chi-square distributions, for which moments may be derived explicitly.  We then use this new class of \textit{noncentral split chi-square distributions} to propose a new routine for calculating the optimal splitting ratio for the split likelihood ratio test based on the dimensionality of the tested null and alternative hypotheses (Section~\ref{limitsection}). Furthermore,  we use this new class of distributions to conduct numerical experiments that analyze the power and the optimal choice of the splitting ratio for the split likelihood ratio (Section~\ref{simulations}).  The simulations suggest, in particular, that while in lower dimensional settings an even split performs well, in higher dimensions a lower splitting ratio is  advantageous and our proposed new splitting ratio significantly improves power. \change{In an experiment with factor analysis models, we demonstrate that our proposal may also lead to a power gain in irregular settings}.  This and other findings are further discussed in the concluding Section~\ref{discussion}.

\paragraph*{Notation.} In the remainder, the symbols $\overset{P}{\rightarrow}$, $\overset{a.s.}{\longrightarrow}$ and $\overset{\mathcal{D}}{\rightarrow}$ stand for convergence in probability, almost sure convergence and convergence in distribution, respectively.  The stochastic Landau symbol $o_P(1)$ indicates convergence to zero in probability.  If not stated otherwise, the limits refer to $n \rightarrow \infty$. \change{With $\mathrm{Id}$ we denote the $d\times d$ identity matrix, and  $\mathcal{N}_d(0,\mathrm{Id})$ is the standard normal distribution in $\mathbb{R}^d$. For a given vector $X \in \R^d$ we denote the vector of its first $p$ components by $X_{[p]}$.}

\section{Background on the split likelihood ratio test}\label{background}


Let $\{P_{\theta}: \theta \in \Theta\}$ be a given (parametric) statistical model, with parameter space \change{$\Theta \subset \R^d$.}
The distributions $P_{\theta}$ are assumed to be dominated by a measure $\mu$, and we write $p_{\theta}$ for the $\mu$-density of $P_{\theta}$.  Given an i.i.d.~sample $X_1,\dots,X_n$ from an unknown distribution $P_{\theta}$ in the model, we are interested in testing 
\begin{equation}
\label{eq:testing}
H_0: \theta \in \Theta_0 \quad\text{versus}\quad H_1:\theta\in\Theta\setminus\Theta_0
\end{equation}
for a subset $\Theta_0\subset\Theta$.  Universal inference solves this problem by appealing to a likelihood ratio, however, one that is built using data splitting.

To split the data, one chooses a fraction $m_0 \in (0,1)$ and partitions the $n$ data points into two disjoint subsets $D_0 = \{ X_{1,0},\ldots,X_{\lfloor m_0 n \rfloor,0} \}$ and $D_1 = \{ X_{1,1},\ldots,X_{\lceil m_1 n \rceil,1}\}$, where  $m_1\equiv 1-m_0$.  In order to lighten notation in subsequent derivations, we simply write $m_0 n $ for $\lfloor m_0 n \rfloor$ and $m_1 n $ for $\lceil m_1 n \rceil$.   Let 
\[
\ell_k(\theta) = \sum_{i =1}^{m_k n} \log p_{\theta}(X_{i,k}), \quad k=0,1,
\]
be the log-likelihood functions based on $D_0$ and $D_1$, respectively.  Let $\widehat{\theta}_{n,0} := \text{argmax}_{\theta \in \Theta_0} \ell_{0}(\theta)$ be the maximum likelihood estimator (MLE) of $\theta$ under $H_0$ and based on $D_0$.  Furthermore, let $\widehat{\theta}_{n,1}:= \text{argmax}_{\theta \in \Theta} \ell_{1}(\theta)$ be the MLE of $\theta$ under the full model and based on $D_1$.  Now the split likelihood ratio statistic is defined as
\begin{equation}\label{split likelihood ratio stat}
	\Lambda_n:=2\left(\ell_{0}(\widehat{\theta}_{n,1})-\ell_{0}(\widehat{\theta}_{n,0})\right).
\end{equation}
As shown in \cite{universal}, an application of Markov's inequality yields for any $\alpha\in(0,1)$ \change{that under $H_0: \theta \in \Theta_0$ we have
\begin{align*}
   P_{\theta}\big(\Lambda_n > -2 \log (\alpha)\big) &\leq \alpha \E_{\theta}\Big[\tfrac{\prod_{i=1}^{m_0n} p_{\widehat{\theta}_{n,1}}(X_{i,0})}{ \prod_{i=1}^{m_0n} p_{\widehat{\theta}_{n,0}}(X_{i,0})}\Big] \\
   &\leq \alpha \E_{\theta}\Big[\E_{\theta}\Big[\tfrac{\prod_{i=1}^{m_0n} p_{\widehat{\theta}_{n,1}}(X_{i,0})}{ \prod_{i=1}^{m_0n} p_{\theta}(X_{i,0})}\Big|D_1\Big]\Big] \leq \alpha.
\end{align*}
Here, we used the fact that $\widehat{\theta}_{n,1}$ is fixed when we condition on $D_1$, and for any fixed $\theta^* \in \Theta$ it holds that 
\begin{equation*}
   \E_{\theta}\Big[\tfrac{\prod_{i=1}^{m_0n} p_{\theta^*}(X_{i,0})}{\prod_{i=1}^{m_0n}p_{\theta}(X_{i,0})}\Big] \leq \int \prod_{i=1}^{m_0n}p_{\theta^*}(X_{i,0}) =1.
\end{equation*}
Therefore,} the decision rule 
\begin{equation}\label{test decision}
	\text{reject } H_0 \ \text{ if } \ \Lambda_n>-2\log\left(\alpha\right)
\end{equation}
constitutes a valid level $\alpha$ test.  Notably, this \emph{split likelihood ratio test} (SLRT) holds level $\alpha$ in finite samples and without any regularity conditions.  

\begin{remark}
The MLE $\widehat{\theta}_{n,1}$ could be replaced by any other estimator and the test would continue to be valid.  While this may be of interest for computational reasons, we \change{focus in the following on the asymptotically efficient maximum likelihood estimator. However, the analysis can be extended in a similar fashion for any asymptotic linear estimator, see Remark \ref{asymplinear}.} 
\end{remark}

In the following, we derive the asymptotic  distribution of the split likelihood ratio $\Lambda_n$ and use it to study the power of the SLRT and the impact of the splitting ratio $m_0$.  Our calculation of the limiting distribution of $\Lambda_n$ is couched in the classical framework of local alternatives in models that are differentiable in quadratic mean.


		
\section{Asymptotic theory for the SLRT}\label{limitsection}

Let $\theta_0$ be a point in the interior of $\Theta$.  Assume that  $\theta_0\in\Theta_0$ and  define the sequence of parameters $\theta_n=\theta_0 + h/\sqrt{n}$ for a choice of $h\in\mathbb{R}^d$.  Suppose then that for each (large)  $n$ we are given an i.i.d.~sample of size $n$ from the local alternative $P_{\theta_n}$.  Suppose further that the considered model possesses the usual smoothness properties that lead to chi-square limits for the ordinary likelihood ratio, see, e.g., \cite{vdv}.  Specifically, we assume that:
\begin{itemize}
\item [(A1)] The model $\{P_{\theta}: \theta \in \Theta\}$ is differentiable in quadratic mean at $\theta_0$, with derivative (i.e., score function) $\dot{\ell}_{\theta_0}$.  Its Fisher information $\E_{\theta_0}[\dot{\ell}_{\theta_0}\dot{\ell}_{\theta_0}^T]= I(\theta_0)$ is nonsingular, and there exists a measurable function $\dot{\ell}$ with \linebreak $\E_{\theta_0}[\dot{\ell}^2]<\infty$ such that \begin{equation*}
   |\log p_{\theta_1}(x)- \log p_{\theta_2}(x)|\leq \dot{\ell}(x) \Vert \theta_1 - \theta_2 \Vert
\end{equation*}  
for all $\theta_1,\theta_2$ in a neighborhood of $\theta_0$.
 \item [(A2)] The maximum likelihood estimators $\widehat{\theta}_{n,0}$ and $\widehat{\theta}_{n,1}$ are consistent estimators \change{for  $\theta_0$} under $P_{\theta_0}$.
 \item [(A3)] The local parameter spaces $H_{n}:=\sqrt{n}(\Theta_0-\theta_0)$ converge to a set $H_0$.
 \end{itemize}

\change{Assumption (A3) uses the notion of convergence of sets in the sense of \cite{vdv}, that is, the local parameter spaces $H_n$ converge to the limit hypothesis $H_0$, if the set $H_0$ 
\begin{enumerate}
    \item[a)] is the set of all limits $\lim h_n$ of converging sequences with $h_n \in H_n$ and
    \item[b)] contains all limits $\lim_{i \to \infty} h_{n_i}$ of converging sequences with $h_{n_i} \in H_{n_i}$.
\end{enumerate}
Further, we} note that the set convergence is guaranteed to hold when $\Theta_0$ is defined by polynomial equations and inequalities, in which case the limit $H_0$ is the tangent cone of $\Theta_0$ at $\theta_0$  \citep{drton:2009}. 



\subsection{Asymptotic distribution} 
 
\begin{theorem}{(Asymptotic distribution of the split likelihood ratio statistic)} \label{theorem1}\\ Suppose the considered statistical model $\{P_{\theta}: \theta \in \Theta\}$ satisfies  assumptions (A1)-(A3). Then under $P_{\theta_n}$ with  $\theta_n=\theta_0 + h/\sqrt{n}$, the split likelihood ratio statistic from \eqref{split likelihood ratio stat} satisfies
\begin{equation*}
		\Lambda_n \;\overset{\mathcal{D}}{\longrightarrow}\; \Vert X + \sqrt{m_0}I(\theta_0)^{1/2}h-I(\theta_0)^{1/2} H_0\Vert^2 - \Vert X- \sqrt{\tfrac{m_0}{m_1}} Y \Vert^2,
\end{equation*}
    where $X,Y \sim \change{\mathcal{N}_d(0,\mathrm{Id})}$ independent and $\Vert x - H_0 \Vert=\inf_{h \in H_0} \Vert x - h \Vert$.
	\end{theorem}

	\begin{proof}
	The proof is based on classical local asymptotic normality results. As shown in Theorem 7.12 of \cite{vdv}, our assumption (A1) implies the existence of the Fisher information and the uniform approximation
	\begin{equation} \label{step1}
		\sup_{\|h\|\le M_n} \left| \log \prod_{i=1}^n    \frac{p_{\theta_0+h/\sqrt{n}}}{p_{\theta_0}}(X_i)- \frac{1}{\sqrt{n}}\sum_{i=1}^nh^T\dot{\ell}_{\theta_0}(X_i)+\frac{1}{2}h^T I(\theta_0)h\right| = o_{P_{\theta_0}}(1)
	\end{equation}
	for $M_n$ a slowly diverging sequence in $\mathbb{R}$.  Via the results collected in \cite{vdv}, assumption (A2) implies that consistent MLEs are $\sqrt{n}$-consistent, which entails that both our split sample MLEs $\hat\theta_{n,0}$ and $\hat\theta_{n,1}$ are $\sqrt{n}$-consistent under $P_{\theta_0}$.
	
	Define $\widehat{\psi}_{n,1}:=\sqrt{m_1 n}(\widehat{\theta}_{n,1}-\theta_0)$, $G_{n,0}:=\frac{1}{\sqrt{m_0 n}}\sum_{i=1}^{m_0 n} \dot{\ell}_{\theta_0}(X_{i,0})$ and $ G_{n,1}:=\frac{1}{\sqrt{m_1 n}}\sum_{i=1}^{m_1 n} \dot{\ell}_{\theta_0}(X_{i,1})$.  Let $B(M_n)=\{h\in\mathbb{R}^d:\|h\|\le M_n\}$ be the ball of radius $M_n$.
	Similarly to the proof of Theorem 16.7 in \cite{vdv} but accounting for the split sample, we obtain from \eqref{step1} and the $\sqrt{n}$-consistency of $\hat\theta_{n,0}$ and $\hat\theta_{n,1}$ that
	\begin{align*}
		\Lambda_n&=2\left(\ell_{0}(\widehat{\theta}_{n,1})-\ell_{0}(\widehat{\theta}_{n,0})\right) \\
		&=2\left(\log \prod_{i=1}^{m_0 n}\frac{p_{\theta_0+ \widehat{\psi}_{n,1}/\sqrt{m_1n}}}{p_{\theta_0}}(X_{i,0})-\sup_{h \in H_{m_0n}}\log \prod_{i=1}^{m_0 n}\frac{p_{\theta_0 + h/\sqrt{m_0n}}}{p_{\theta_0}}(X_{i,0}) \right) \\
		&=2\Bigg(\sqrt{\frac{m_0}{m_1}} \widehat{\psi}_{n,1}^TG_{n,0} - \frac{m_0}{2m_1} \widehat{\psi}_{n,1}^T I(\theta_0)  \widehat{\psi}_{n,1}\\
		&\qquad \quad - \sup_{h \in H_{m_0n}\cap B(M_n)}\left(h^T G_{n,0} - \frac{1}{2} h^T I(\theta_0) h\right)\Bigg) + o_{P_{\theta_0}}(1) \\
		&=\Vert I(\theta_0)^{-1/2}G_{n,0}-I(\theta_0)^{1/2}[H_{m_0n}\cap B(M_n)]\Vert^2 \\
		& \qquad \quad -\Vert I(\theta_0)^{-1/2}G_{n,0} - \sqrt{\frac{m_0}{m_1}}I(\theta_0)^{1/2} \widehat{\psi}_{n,1}\Vert^2 + o_{P_{\theta_0}}(1).
	\end{align*}
	By Theorem 5.39 in \cite{vdv}, the MLE $\hat\theta_{n,1}$ is asymptotically linear with $\widehat{\psi}_{n,1}=I(\theta_0)^{-1}G_{n,1} + o_{P_{\theta_0}}(1)$.  Hence,
	\begin{align}
	\label{eq:step2}
		\Lambda_n&=\Vert I(\theta_0)^{-1/2}G_{n,0}-I(\theta_0)^{1/2} [H_{m_0n}\cap B(M_n)]\Vert^2 \\
		&\qquad  -\Vert I(\theta_0)^{-1/2}G_{n,0} - \sqrt{\frac{m_0}{m_1}}I(\theta_0)^{-1/2}G_{n,1} \Vert^2 + o_{P_{\theta_0}}(1). \nonumber
	\end{align}
	Now we use Le Cam's Lemmas to show contiguity of $P_{\theta_0}$ and $P_{\theta_n}$. Applying \eqref{step1},
		\begin{align*}
			\log \frac{\text{d}P^{\otimes n}_{\theta_n}}{\text{d}P^{\otimes n}_{\theta_0}}&=\log \left(  \prod_{i=1}^{m_0 n} \frac{p_{\theta_0+h/\sqrt{n}}}{p_{\theta_0}}(X_{i,0})  \prod_{i=1}^{m_1 n} \frac{p_{\theta_0+h/\sqrt{n}}}{p_{\theta_0}}(X_{i,1}) \right) \\&= \frac{1}{\sqrt{n}}\left(\sum_{i=1}^{m_0n} h^T\dot{\ell}_{\theta_0}(X_{i,0})+ \sum_{i=1}^{m_1n} h^T\dot{\ell}_{\theta_0}(X_{i,1})\right)-\frac{1}{2}h^T I(\theta_0)h + o_{P_{\theta_0}}(1).
		\end{align*}
    The central limit theorem yields that under $P_{\theta_0}$,
		\begin{equation*}
			\begin{pmatrix}
				G_{n,0} \\
				G_{n,1} \\
				\log \frac{\text{d}P^{\otimes n}_{\theta_n}}{\text{d}P^{\otimes n}_{\theta_0}}
			\end{pmatrix} \;\overset{\mathcal{D}}{\longrightarrow} \; \change{\mathcal{N}_{2d+1}} \left(\mu,\Sigma\right),
		\end{equation*}
		where
		\begin{equation*}
		    \mu:=\begin{pmatrix}
					0 \\ 0 \\ -\frac{1}{2}h^T I(\theta_0)h
		\end{pmatrix}, \qquad \Sigma:= \begin{bmatrix}
		I(\theta_0) & 0 & \sqrt{m_0} I(\theta_0)h \\
		0 & I(\theta_0) & \sqrt{m_1}I(\theta_0)h \\
		\sqrt{m_0}I(\theta_0)h & \sqrt{m_1}I(\theta_0)h & h^TI(\theta_0)h 
		\end{bmatrix}.
		\end{equation*}
		By Le Cam's first lemma, the probability measures $P_{\theta_0}$ and $P_{\theta_n}$ are thus mutually contiguous and therefore $o_{P_{\theta_0}}(1)$ and $o_{P_{\theta_n}}(1)$ interchangeable.  By Le Cam's third lemma, it follows that under $P_{\theta_n}$ we have
		\begin{equation}\label{step2}
			\begin{pmatrix}
					G_{n,0} \\
				G_{n,1}
			\end{pmatrix} \; \overset{\mathcal{D}}{\rightarrow} \; \change{\mathcal{N}_{2d}} \left(\begin{pmatrix} 
			\sqrt{m_0} I(\theta_0)h \\
			\sqrt{m_1} I(\theta_0)h 
			\end{pmatrix}, \begin{bmatrix}
				I(\theta_0) & 0 \\
				0 &	I(\theta_0) 
		\end{bmatrix}
			\right).
		\end{equation}
       We may now use this joint convergence in \eqref{eq:step2} and arrive at our claim by observing that for any converging sequence of random vectors $X_n \overset{\mathcal{D}}{\rightarrow} X$ and any sequence of converging sets $H_n \rightarrow H$ it holds that
		\begin{equation*}
		    \Vert X_n - H_n \Vert \; \overset{\mathcal{D}}{\rightarrow}\; \Vert X - H \Vert;
		\end{equation*}
		see Lemma 7.13 in \cite{vdv}. Indeed, our assumption (A3) implies the convergence \linebreak $H_{m_0n}\cap B(M_n)\rightarrow H_0$ and our claim follows.
	\end{proof}
 \change{\begin{remark} \label{asymplinear}
     Throughout this work we use the MLE $\widehat{\theta}_{n,1}$ to solve the estimation task on data set $D_1$. Suppose we employ instead a suboptimal, asymptotically linear estimator $\tilde{\theta}_{n,1}$, that is,
    \[ \sqrt{m_1n}(\tilde{\theta}_{n,1}-\theta_0)=\tfrac{1}{\sqrt{m_1n}}I(\theta_0)^{-1}\sum_{i=1}^{m_1n}\tilde{g}(X_{i,1}) + o_{P_{\theta_0}}(1), 
    \]
    with $\mathrm{Var}[\tilde{g}(X_{i,1})]\succeq I(\theta_0)$, where $\succeq$ denotes the Loewner order. Tracing the proof of Theorem \ref{theorem1}, one obtains that
    \begin{equation}\label{eq:limitsubopt}
		\tilde{\Lambda}_n \;\overset{\mathcal{D}}{\longrightarrow}\; \Vert X + \sqrt{m_0}I(\theta_0)^{1/2}h-I(\theta_0)^{1/2} H_0\Vert^2 - \Vert X- \sqrt{\tfrac{m_0}{m_1}} \tilde{Y} \Vert^2,
    \end{equation}
    where $X,\tilde{Y}$ are independent, $X \sim \mathcal{N}_d(0,\mathrm{Id})$ and $\tilde{Y} \sim \mathcal{N}_d(\mu_{\tilde{Y}},V_{\tilde{Y}})$ with $\mu_{\tilde{Y}}\neq 0$ and $V_{\tilde{Y}} \succeq \mathrm{Id}$. In fact, we may assume that $V_{\tilde{Y}}$ is diagonal with diagonal entries $(V_{\tilde{Y}})_{jj}\geq1.$  (Otherwise, apply an orthogonal transformation to $X- \sqrt{\tfrac{m_0}{m_1}} \tilde{Y}$ to diagonalize $V_{\tilde{Y}}$.) The second part of the representation of the asymptotic distribution is thus 
    \begin{equation*}
    \Vert X- \sqrt{\tfrac{m_0}{m_1}} \tilde{Y} \Vert^2 \overset{\mathcal{D}}{=}\sum_{i=1}^d\big(1+\tfrac{m_0}{m_1}(V_{\tilde{Y}})_{ii}\big) Z_i^2 \geq \big(1+\tfrac{m_0}{m_1}\big) \sum_{i=1}^d Z_i^2,   
    \end{equation*}
    where $Z_i \sim \mathcal{N}(-\sqrt{\tfrac{m_0}{m_1}}(\mu_{\tilde{Y}})_i,1)$ are independent.  Now, $\sum_{i=1}^d Z_i^2 \sim \chi^2_d(\lambda)$ with  noncentrality parameter $\lambda>0$.  Since $\chi^2_d(\lambda)$ is stochastically larger than a (central) $\chi^2_d$-distribution, we obtain from \eqref{eq:limitsubopt} that the limiting distribution when using a suboptimal estimator is stochastically smaller than the limit distribution of Theorem \ref{theorem1} with the MLE.  Using a suboptimal estimator on $D_1$ thus leads to a decrease in power for the SLRT.
    \end{remark}}
    

    In the sequel, we investigate properties of the limiting distribution of the split likelihood ratio statistic in the smooth case, where the original null hypothesis is a $k$-dimensional smooth manifold and the limiting set $H_0$ is thus a $k$-dimensional tangent space. We start by introducing the arising new class of distributions, \textit{noncentral split chi-square distributions}, that depends on four parameters, the dimension of the parameter space, the dimension of the null hypothesis, the splitting ratio, and a noncentrality parameter.  
    
    \begin{definition}{(Noncentral split chi-square distribution)} \\
     Let $d\in\mathbb{N}$, \change{$p\in\{0,\dots,d\}$}, $\delta\ge 0$, and $m_0\in(0,1)$.  
      The $d$-dimensional \textit{noncentral split chi-square distribution} with $p$ degrees of freedom, noncentrality parameter $\delta$ and splitting ratio $m_0$, denoted $\textnormal{split}_{m_0}\textnormal{-}\chi^{2}_{p,d}(\delta)$, is the distribution of 
      \[
      \Vert X_{[p]} +\sqrt{m_0}h \Vert^2 - \Vert X- \sqrt{\tfrac{m_0}{1-m_0}} Y \Vert^2 \;\sim\; \textnormal{split}_{m_0}\textnormal{-}\chi^{2}_{p,d}(\delta),
      \]
      where $X,Y\sim\mathcal{N}_d(0,\mathrm{Id})$ \change{independent} and $h \in \R^p$ such that $h^Th=\delta$.
    \end{definition}
        
    We emphasize that the noncentral split chi-square distribution is well-defined in that it depends on the vector $h$ only through its norm $\delta$.  This follows from the invariance of the standard normal distribution of $X$ and  $Y$ under orthogonal rotations analogously to the classical noncentral chi-square distribution.  
    
    While the classical noncentral chi-square distribution is the distribution of the squared distance from a standard normal vector to some fixed point in the space, the noncentral split chi-square distribution is the distribution of the difference of two squared distances. The first part is the squared distance from a standard normal vector to a fixed point in the space. However, a second part arises from splitting the data into two subsets, namely, the squared distance of two independent standard normal vectors scaled according to the splitting ratio. Notice that the two parts are not independent.
    
    In the following we calculate the first moments of this new class of distributions. In Section~\ref{subsec:optimal} 
    we use the calculated moments to approximate the noncentral split chi-square distribution and thus the asymptotic behavior of the SLRT.
	\begin{corollary}{(Moments of the noncentral split chi-square distribution)} \label{cor2} \\
		Let $Z\sim\textnormal{split}_{m_0}\textnormal{-}\chi^{2}_{p,d}(\delta)$. Then
		\begin{enumerate}
			\item 
			$
			\E[Z]=	p-d-d\frac{m_0}{1-m_0}+m_0 \delta,
			$
			\item
			$
			\mathrm{Var}[Z]=	2(d-p) + 4d \frac{m_0}{1-m_0} + 2d \frac{m_0^2}{(1-m_0)^2} + 4m_0 \delta. 
			$
		\end{enumerate}
	\end{corollary}
	
	\begin{proof}
	Let  
	\[
	\epsilon \sim \mathcal{N}_{2d}\left(0,\begin{bmatrix} m_0^{-1}\mathrm{Id} & 0 \\ 0 & m_1^{-1} \mathrm{Id} \end{bmatrix}\right),
	\]
	and define \[
	\mu  :=  \begin{pmatrix}
	h \\ 0 \\ h \\ 0
		\end{pmatrix}, \quad
		A := \begin{bmatrix}
		0 & 0 & \mathrm{Id_p} & 0  \\
		0 & -\mathrm{Id_k} & 0 & \mathrm{Id_k} \\
		\mathrm{Id_p} & 0 & -\mathrm{Id_p} & 0 \\
		0 & \mathrm{Id_k} & 0 & -\mathrm{Id_k}
		\end{bmatrix},
	\]
	with $h \in \R^p$ such that $h^Th=\delta$. Then the quadratic form $m_0(\epsilon+\mu)^T A (\epsilon+\mu)$ follows a
	$\textnormal{split}_{m_0}\textnormal{-}\chi^{2}_{p,d}(\delta)$ distribution and we can use properties of quadratic forms to calculate moments of the noncentral split chi-square distribution.

	Using $\E [(\epsilon+\mu)^T A (\epsilon+\mu)]= \mathrm{tr}[A \Sigma] + \mu^T A \mu $, a short calculation yields the claim for the expectation and the claimed variance follows via Var$[(\epsilon+\mu)^T A (\epsilon+\mu)]=2\, \mathrm{tr}[A \Sigma A \Sigma] + 4 \mu^T A \Sigma A \mu $.
	\end{proof}
	\begin{remark}
    Higher moments can be calculated via the cumulants $\kappa_n(\epsilon^T A \epsilon)=2^{n-1}(n-1)!\,\mathrm{tr}[A^n]$ with the following formulas for moments of quadratic forms:
    \begin{enumerate}
        \item[1.] $\E[(\epsilon^T A \epsilon)^1]=\kappa_1$.
        \item[2.] $\E[(\epsilon^T A \epsilon)^2]=\kappa_1^2+\kappa_2$.
        \item[3.] $\E[(\epsilon^T A \epsilon)^3]=\kappa_1^3+3\kappa_1\kappa_2+\kappa_3$.
        \item[4.] $\E[(\epsilon^T A \epsilon)^4]=\kappa_1^4+6\kappa_1^2\kappa_2+3\kappa_2^2+4\kappa_1\kappa_3+\kappa_4$.
    \end{enumerate}
    Formulas for moments up to order ten can be found in \cite{kendall}.
    \end{remark}

    Due to the rotational invariance of the standard normal distribution, we may study the limit of the SLRT in the smooth case, where the limiting hypothesis is a $k$-dimensional tangent space, by simply assuming that $I(\theta_0)^{1/2}H_0$ is a coordinate subspace, i.e., $I(\theta_0)^{1/2}H_0 =  \{0\}^p \times \R^k$  with $d=p+k$. 
	
	\begin{corollary} \label{cor1}
		If the rotated limiting hypothesis $I(\theta_0)^{1/2}H_0=\{0\}^p \times \R^k$ is a coordinate subspace, then the asymptotic distribution from Theorem \ref{theorem1} follows a d-dimensional \textit{noncentral split chi-square distribution} with $p$ degrees of freedom, noncentrality parameter $\tilde h^T \tilde h$ and splitting ratio $m_0$. That is
	\begin{equation*}
		\Lambda_{\infty} \overset{\mathcal{D}}{=}	\Vert X_{[p]} +\sqrt{m_0}\tilde h_{[p]} \Vert^2 - \Vert X- \sqrt{\tfrac{m_0}{m_1}} Y \Vert^2 \;\sim\; \textnormal{split}_{m_0}\textnormal{-}\chi^{2}_{p,d}(\tilde h_{[p]}^T \tilde h_{[p]}),
	\end{equation*}
	with $X,Y \sim \mathcal{N}_d(0,\mathrm{Id})$ independent and $\tilde h=[I(\theta_0)^{1/2}h]_{[p]}$.	
	\begin{proof}
		We look at the first part of the limiting distribution from Theorem \ref{theorem1}. With $X \sim \mathcal{N}_d(\sqrt{m_0}I(\theta_0)^{1/2}h, \mathrm{Id})$  we have 
		\[
		\Vert X -I(\theta_0)^{1/2}H_0\Vert^2= \inf_{\theta \in \R^k} \left(X- \begin{pmatrix}
			0 \\ \theta
		\end{pmatrix}\right)^T  \left(X- \begin{pmatrix}
		0 \\ \theta
		\end{pmatrix}\right) = X_{[p]}^T  X_{[p]},
		\]
        and the claim follows immediately.
	\end{proof} 

	\end{corollary}
    
	\begin{remark}
	Under the null hypothesis, the limiting distribution of the split likelihood ratio test statistic reduces to the following difference of dependent (scaled) chi-square distributions 
	\begin{equation}\label{limhyp}
	\Lambda_n \;\overset{\mathcal{D}}{\rightarrow} \;\Vert X_{[p]} \Vert^2 - \Vert X- \sqrt{\tfrac{m_0}{m_1}} Y \Vert^2,
    \end{equation}
	with $X,Y \sim \mathcal{N}_d(0,\mathrm{Id})$ independent. 
	\end{remark}
	
	The limiting null distribution in \eqref{limhyp} clearly shows the asymptotic difference between the LRT and the SLRT. For the SLRT, a new second term arises in the limit that behaves like a scaled chi-square distributed random variable where the scaling factor depends only on  the chosen splitting ratio. Furthermore, looking at \change{Corollary \ref{cor2}}, the limiting distribution has a negative expectation under the null hypothesis.

    The SLRT uses the conservative  critical value $-2\log(\alpha)$ that is universally valid but whose adoption may come with a loss of power.  In Figures \ref{empiricalquant} and \ref{empiricalquant2} we illustrate the source of this loss of power in different settings by comparing the universal threshold (\texttt{SLRT}) of the SLRT with 
    (simulated) quantiles of a split chi-square distribution (\texttt{Asym}), the limiting distribution  under the null hypothesis \eqref{limhyp}. \change{The difference between the universal threshold and the quantile of the limiting distribution is smaller for lower significance level $\alpha$ and thus, the power loss, which stems from employing an universal threshold, is less noteworthy for low significance levels.}  Furthermore, the universal threshold is asymptotically more accurate for smaller splitting ratios. Moreover, we observe that \change{using} the universal threshold is asymptotically less precise for higher dimensions of the parameter space and for higher dimensional hypotheses. In Section \ref{simulations1} we further analyze the power loss from using the universal threshold in simulations.

    \begin{figure}[t]
		\centering
		\includegraphics[width=0.75\linewidth]{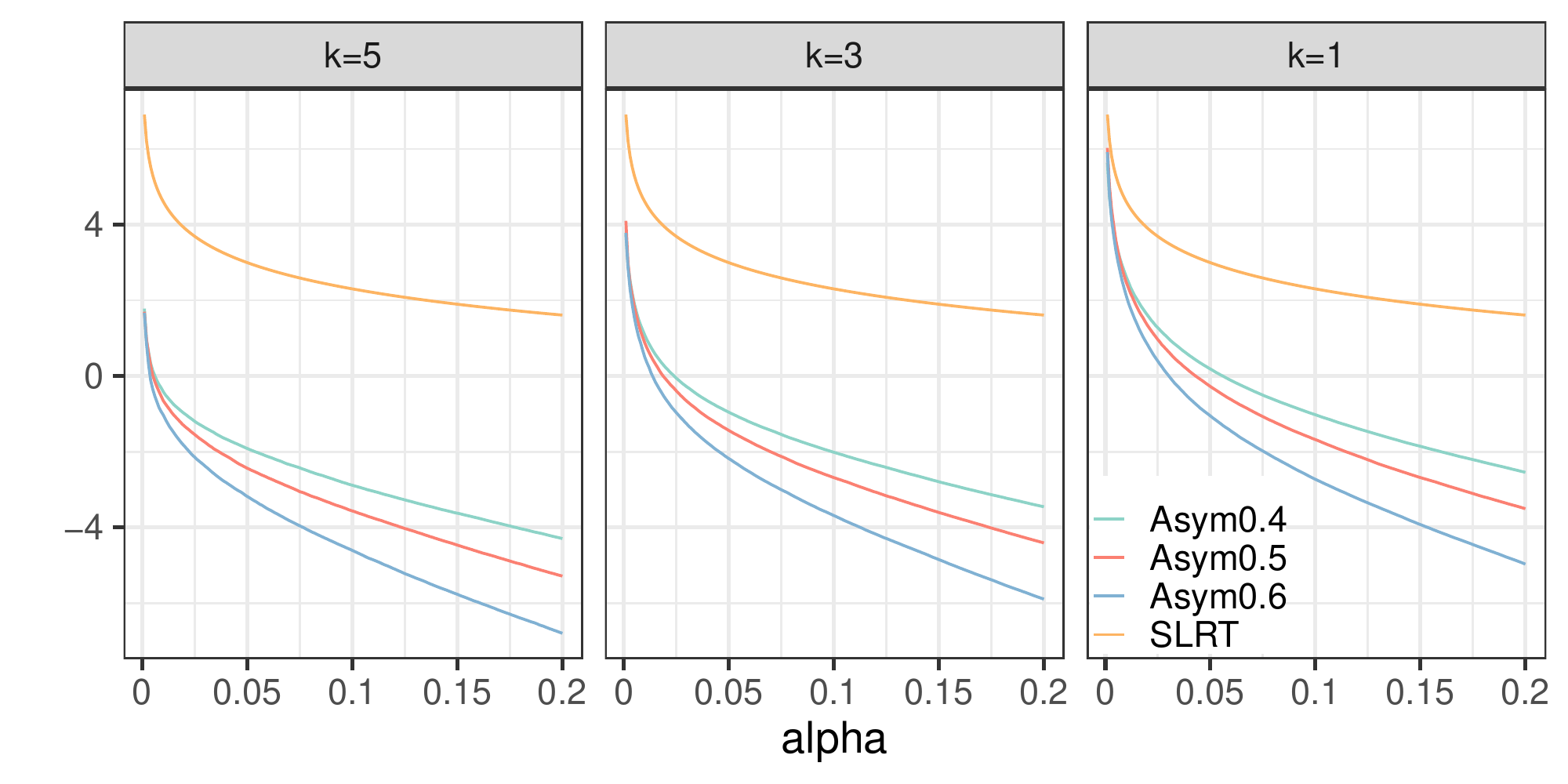}
		\caption{Quantile of $\textnormal{split}_{m_0}\textnormal{-}\chi^{2}_{p,6}$ compared to the universal threshold.}\label{empiricalquant}
		
		\centering
		\includegraphics[width=0.78\linewidth]{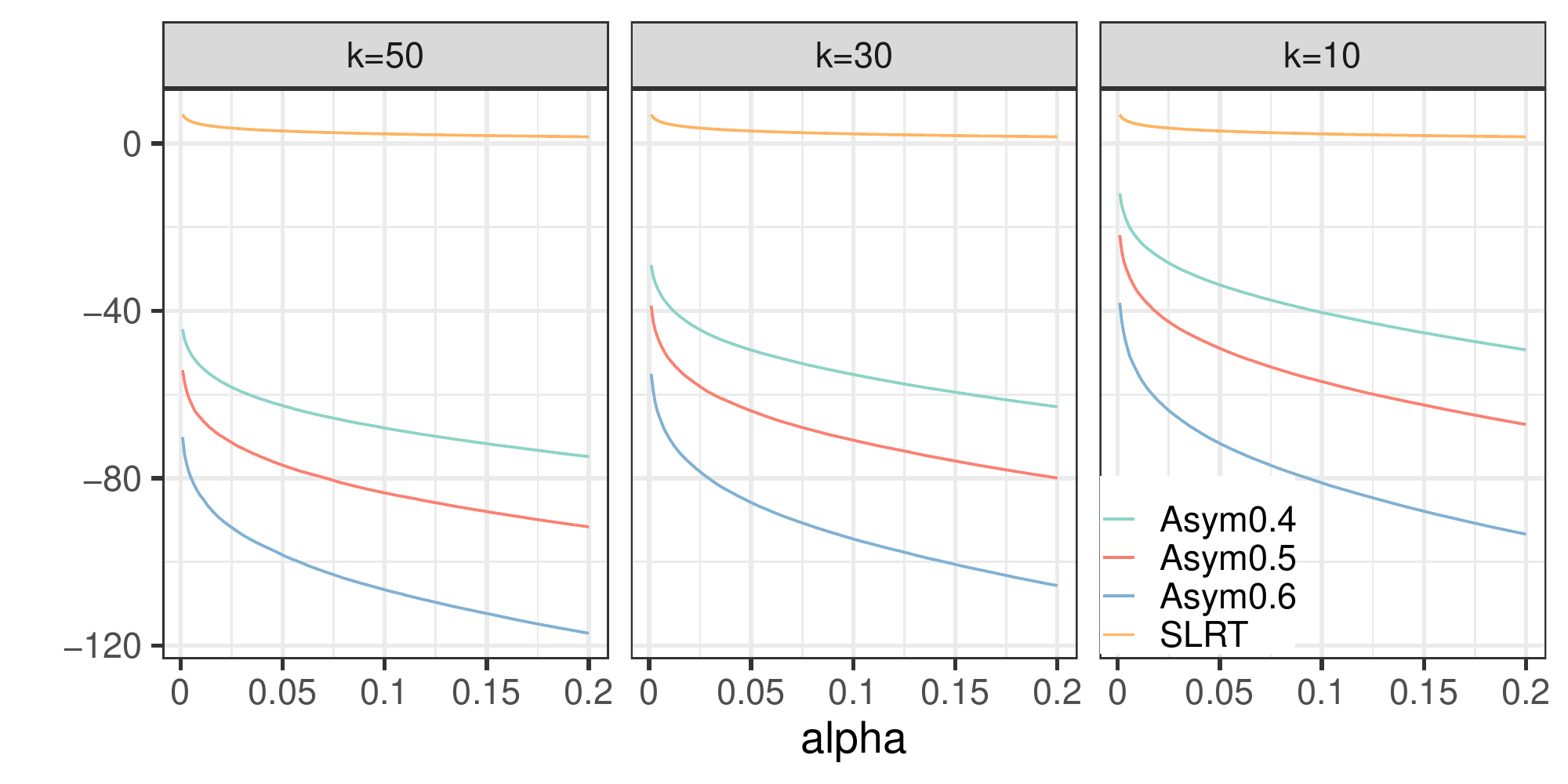}
		\caption{Quantile of $\textnormal{split}_{m_0}\textnormal{-}\chi^{2}_{p,60}$ compared to the universal threshold.}\label{empiricalquant2}
    \end{figure}
    
    \subsection{Optimal splitting ratio}
    \label{subsec:optimal}
    
	The main advantage of the SLRT over classical likelihood methods is its flexibility for settings where asymptotic distributions are difficult to obtain. This flexibility that stems from using only the general Markov inequality to control the type I error comes at the price of a potential loss of power.  In the smooth setting of Theorem \ref{theorem1}, we could improve the asymptotic power of the SLRT by using quantiles from the calculated asymptotic distribution, but such an asymptotic SLRT is not of practical relevance as the testing problem could then be better solved using the standard LRT, see Section \ref{simulations1}.  
	
	Instead, our focus will remain on the SLRT with its conservative critical value $-2\log(\alpha)$, and our goal is to provide a new method for choosing the splitting ratio $m_0$ that helps retain power. The idea behind our proposed method is simple.  Having access to the asymptotic distribution of the split likelihood ratio, the noncentral split chi-square distribution, we choose the splitting ratio that achieves the highest (asymptotic) power. Given both the dimensions of the null and alternative hypotheses and a significance level $\alpha$, we minimize the cumulative distribution function of the noncentral split chi-square distribution at \change{$-2\log(\alpha)$} with respect to the splitting ratio. To achieve a meaningful and comparable power, we propose to scale the unknown noncentrality parameter such that the best-performing method achieves a power of $0.8$. \change{We note that this tuning parameter of our method can easily be adapted if practitioners prefer to choose the best-performing method in different power levels. However, in our experience the effect of this tuning parameter on the performance is negligible across reasonable power levels.}
	
	As we are still lacking dedicated numerical routines to evaluate the
	cumulative distribution function of the noncentral split chi-square distribution, we use Monte Carlo approximations via repeated sampling from a noncentral split chi-square distribution.  Using these approximations, we then choose the best-performing splitting ratio over a fine grid. 
	
	\begin{figure}[t]
	\centering
    \includegraphics[width=0.49\linewidth]{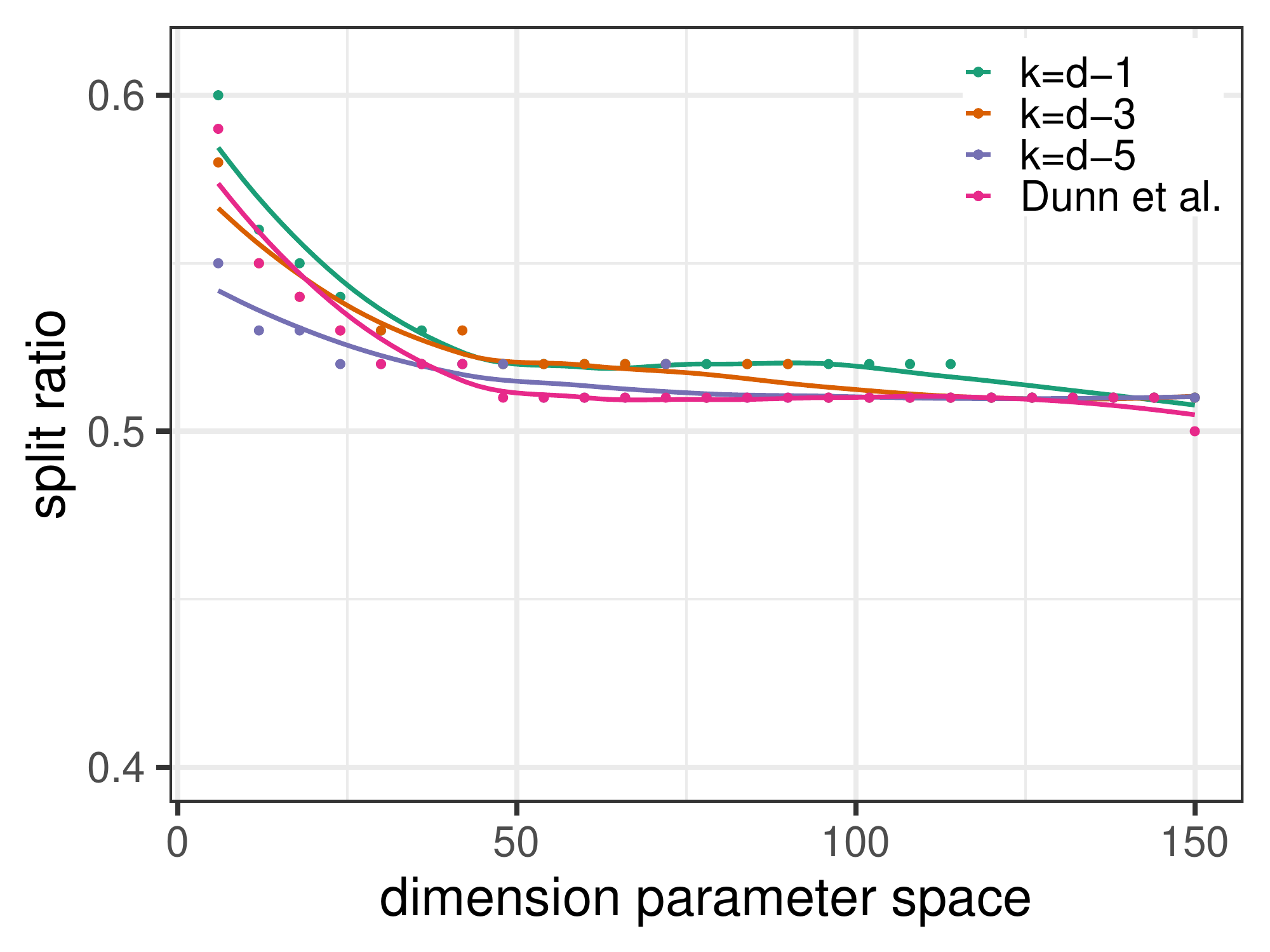}	    \includegraphics[width=0.49\linewidth]{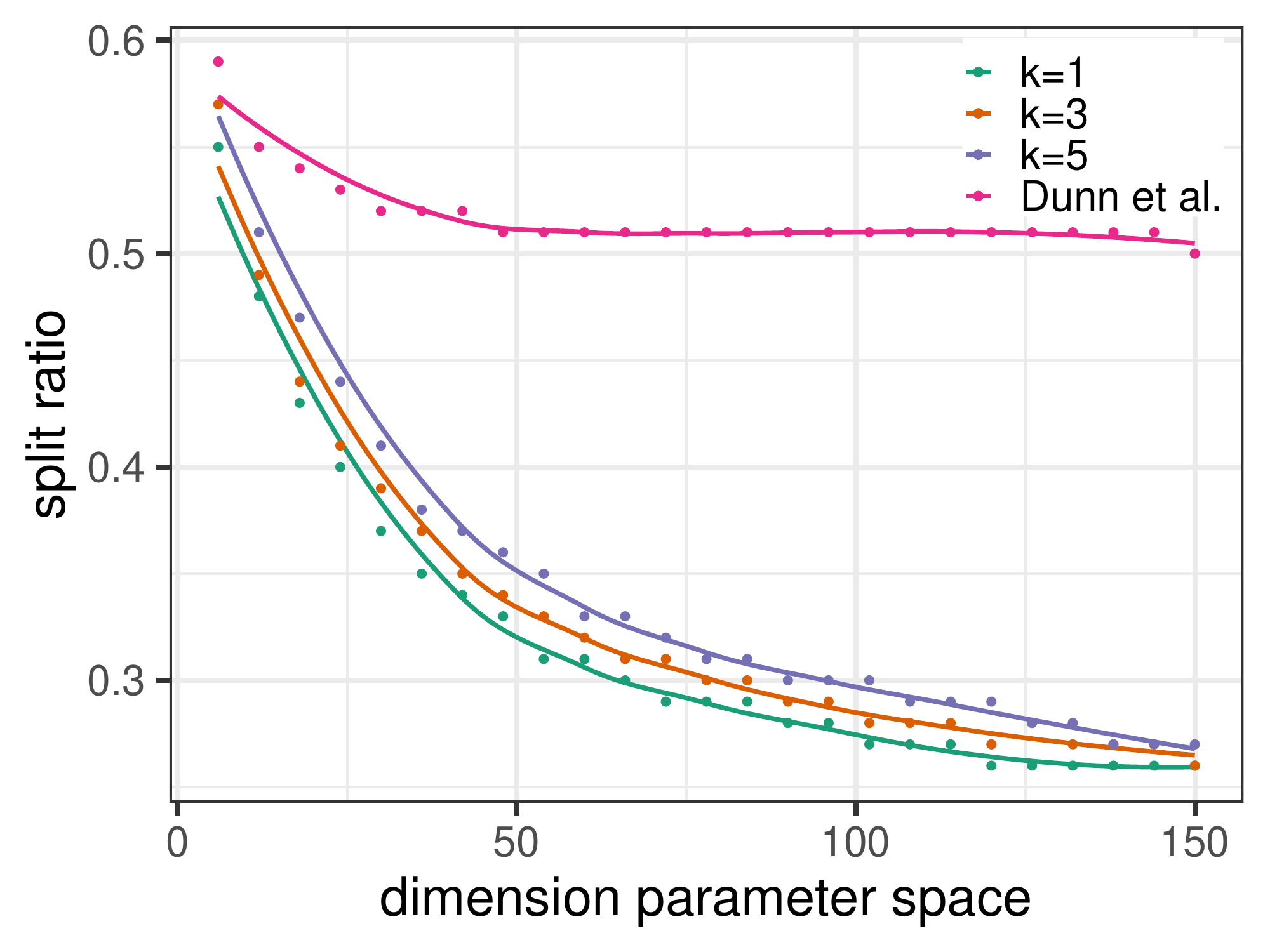}
	\includegraphics[width=0.49\linewidth]{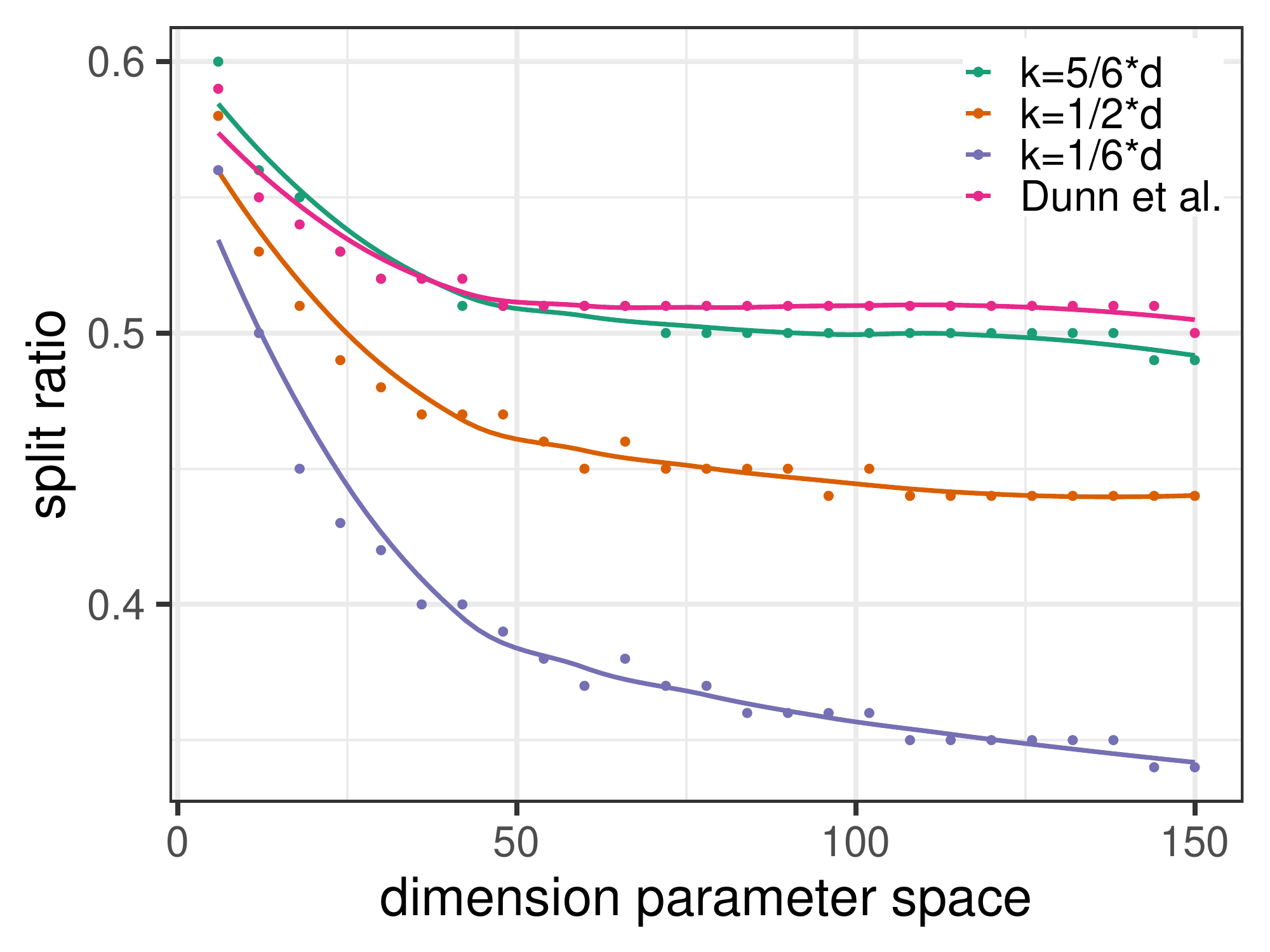}
	\caption{Optimal splitting ratio against dimension of parameter space.}\label{empiricalsplit}
	\end{figure}
	
	\change{Figure \ref{empiricalsplit} displays our results for the optimal choice of the splitting ratio against the dimension of the parameter space for different regimes of the null hypothesis space $k$. We observe that the underlying dimension of the null hypothesis $k$ is crucial for the optimal choice of the splitting ratio. For comparison Figure \ref{empiricalsplit} additionally shows the difference to the existing proposal for the choice of the splitting ratio by \citet{universal2}.} 
	To obtain a high power they propose to use the split ratio \begin{equation}\label{propsplit}
	  m_0=1-\frac{\sqrt{4d^2+8d\log(1/\alpha)}-2d}{4\log(1/\alpha)},
	\end{equation} 
	which minimizes the squared radius of the universal inference confidence set for the mean vector of a Gaussian distribution. In contrast to what our results suggest, their proposed split does not depend on the dimension of the null hypothesis.  This dimensionality is, however, incorporated in our proposed choice of the splitting ratio. 
	
	Note that analytically the split \eqref{propsplit}, proposed by \citet{universal2}, converges to $0.5$ for high dimensions $d$. \change{ In a similar fashion, our proposed splitting ratio converges to $0.5$ for high dimensions when testing a fixed number of parameters to be constant, that is $d-k$ is fixed. However, we can see that the optimal splitting ratio varies with the dimension of the tested hypothesis space, thus, e.g. in settings where the number of tested, fixed parameters in the hypothesis grows proportional with the dimension, a splitting ratios below $0.5$ is beneficial, even in the limit.}
	
	\change{
	Further, due to the impact of the noncentrality parameter and the complexity of the limit distribution, it is difficult to determine the optimal splitting ratio analytically. To avoid extensive Monte Carlo approximation for the calculation of the optimal splitting ratio, we additionally propose the following computationally fast alternative.} Instead of using extensive simulations to approximate the noncentral split chi-square distribution, we employ normal approximations with the expectation and variance calculated based on Corollary \ref{cor2}. This then leads to  Algorithm \ref{alg:opt}, which very quickly determines an optimal splitting ratio based on the dimension of the null hypothesis $k$, the dimension of the parameter space $d$, and the significance level $\alpha$ via repeated minimization of values of the standard Gaussian cdf (the \texttt{pnorm} function in \texttt{R}).

    \begin{algorithm}[t]
    \caption{Optimal splitting ratio}\label{alg:opt}
    \begin{algorithmic}
    \State \textbf{Input:} $d$, $p$, $\alpha$
    \State \textbf{initialize} $\delta$ small
    \State power $\gets 0.5$
    \While{power$<0.8$}
        \State exp($m_0$) $\gets$ $p-d-d\frac{m_0}{1-m_0}+m_0 \delta$\Comment{Define functions in split $m_0$}
        \State var($m_0$) $\gets$ $2(d-p) + 4d \frac{m_0}{1-m_0} + 2d \frac{m_0^2}{(1-m_0)^2} + 4m_0 \delta$
        \State target($m_0$) $\gets$ \texttt{pnorm}($-2 \log(\alpha$),exp($m_0$),$\sqrt{\textnormal{var($m_0$)}})$ \vspace{0.07cm}
        \State (min, value) $\gets$ \textbf{minimize} target($m_0$) \textbf{for} $m_0$ \textbf{in} $(0,1)$ \Comment{Optimization} 
        \State power $\gets$ $1-$value
        \State \textbf{increase} $\delta$
    \EndWhile
    \State \textbf{return} min
    \end{algorithmic}
    \end{algorithm}
	
	\change{Our simple algorithm constitutes a fast solution to obtain an optimal splitting ratio based on normal approximations of the limit distribution. Without considering the variance, our algorithm can be interpreted as maximizing the expectation of the limit distribution of the SLRT, see Corollary \ref{cor2}. However, this expectation depends on the distance of the local alternative $\delta$. Thus, to achieve a meaningful choice between the different splitting ratios, we first scaled $\delta$ to represent a setting where the test still has power against the local alternative but the task is nevertheless non-trivial. Furthermore, to get some more intuition for the behaviour of our proposed splitting ratio, Figure \ref{optimasym} provides a visual rule of thumb for the choice of the optimal splitting ratio based on the ratio of the dimension of the hypothesis space and the dimension of the parameter space derived in a high-dimensional setting. We found the impact of the significance level $\alpha$ in this high-dimensional setting to be negligible. For further simplicity and easy computation, this visual rule of thumb can be approximated with the function $m_0=-\exp(-2.7\tfrac{k}{d}-1.05)+0.52$.}
	
	In Section \ref{simulations2} we analyze the performance of both proposed methods, the Monte Carlo method and the normal approximation variant numerically and show that it outperforms the existing proposal by \citet{universal2} in regular settings. \change{Further, in Section \ref{section:powerirregular} we show that this splitting ratio leads to optimal power even in the investigated irregular settings}.
	
	\begin{figure}[t]
	\centering
    \includegraphics[width=0.48\linewidth]{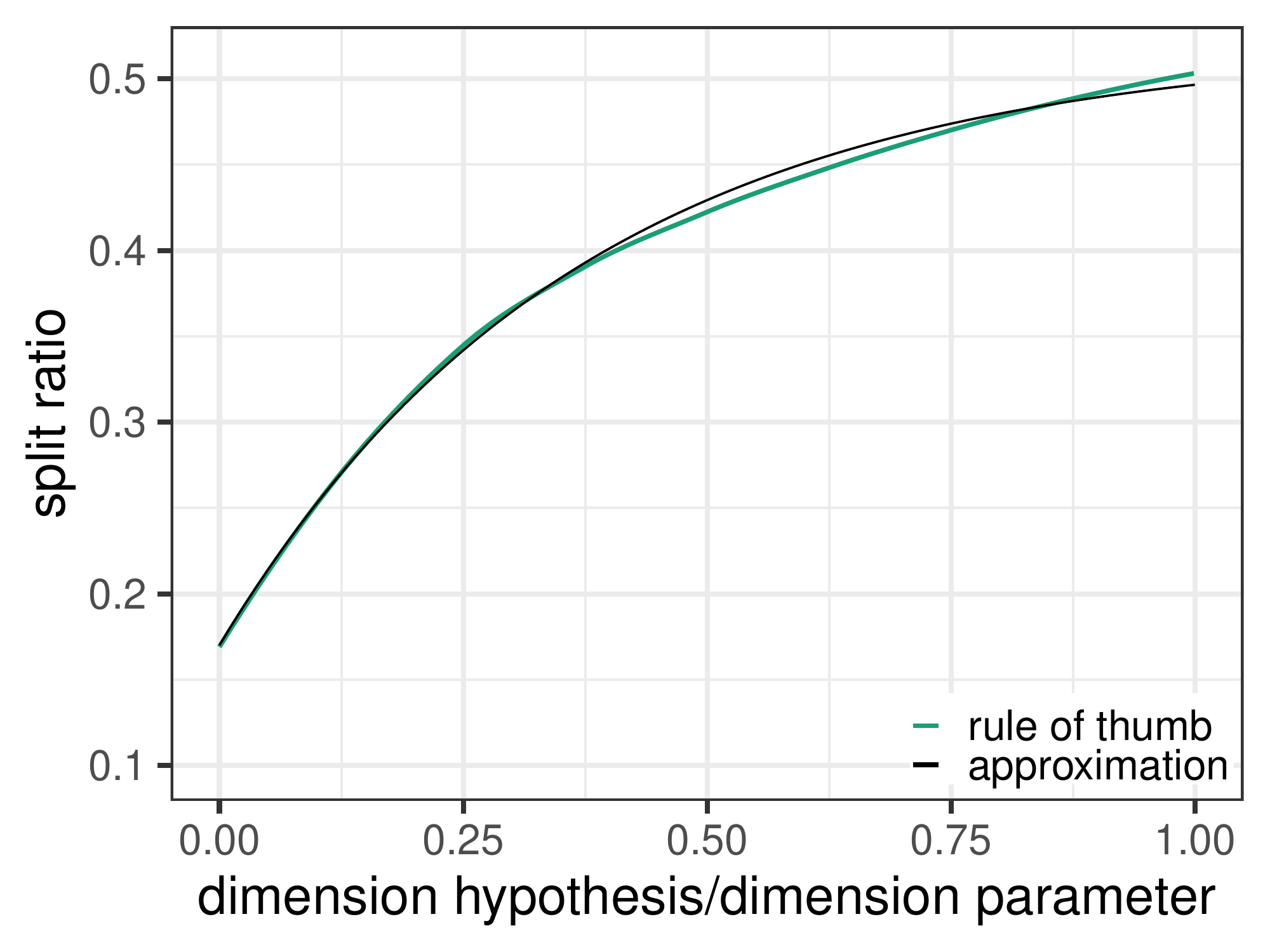}	\change{    
	\caption{Rule of thumb for the choice of the split ratio in high dimensions.}\label{optimasym}}
	\end{figure}


\section{Simulations} \label{simulations}
    We present the results of a simulation study that investigates the asymptotic behavior of the SLRT and compare its performance in different model settings, namely, regular and irregular settings, different dimensions of the parameter space $d$, different dimensions of the null hypothesis space $k$, and different splitting ratios $m_0$. All reported quantities are computed from simulations with 100,000 replications and if not stated otherwise we use the significance level $\alpha=0.05$.

\subsection{Power of the SLRT in regular setting}\label{simulations1}

    \begin{figure}[t]
		\centering
		\includegraphics[width=0.75\linewidth]{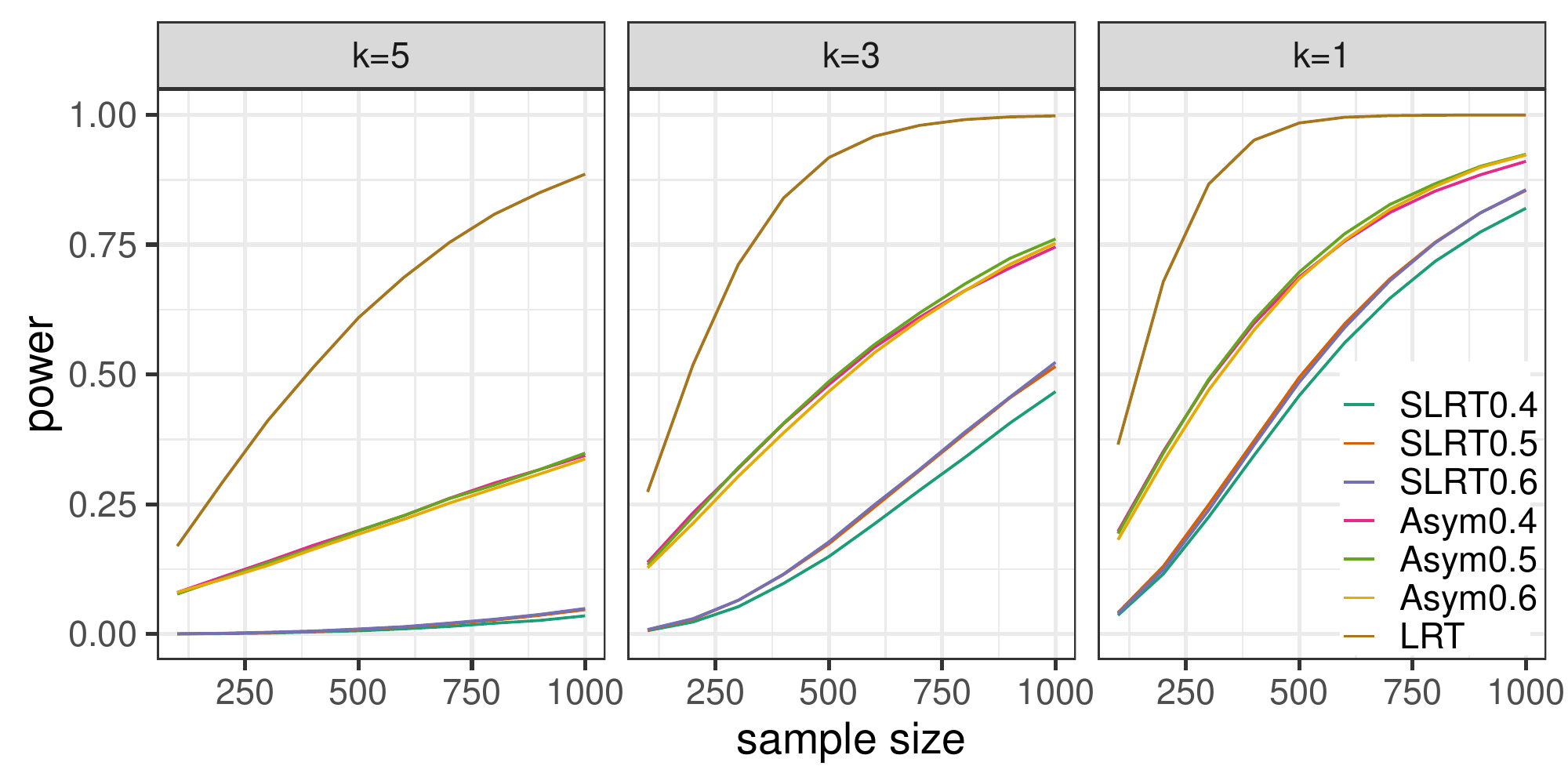}
		\caption{Power against sample size in Gaussian setting with $\theta=0.1$, $d=6$.}  \label{powern}

        \centering
		\includegraphics[width=0.75\linewidth]{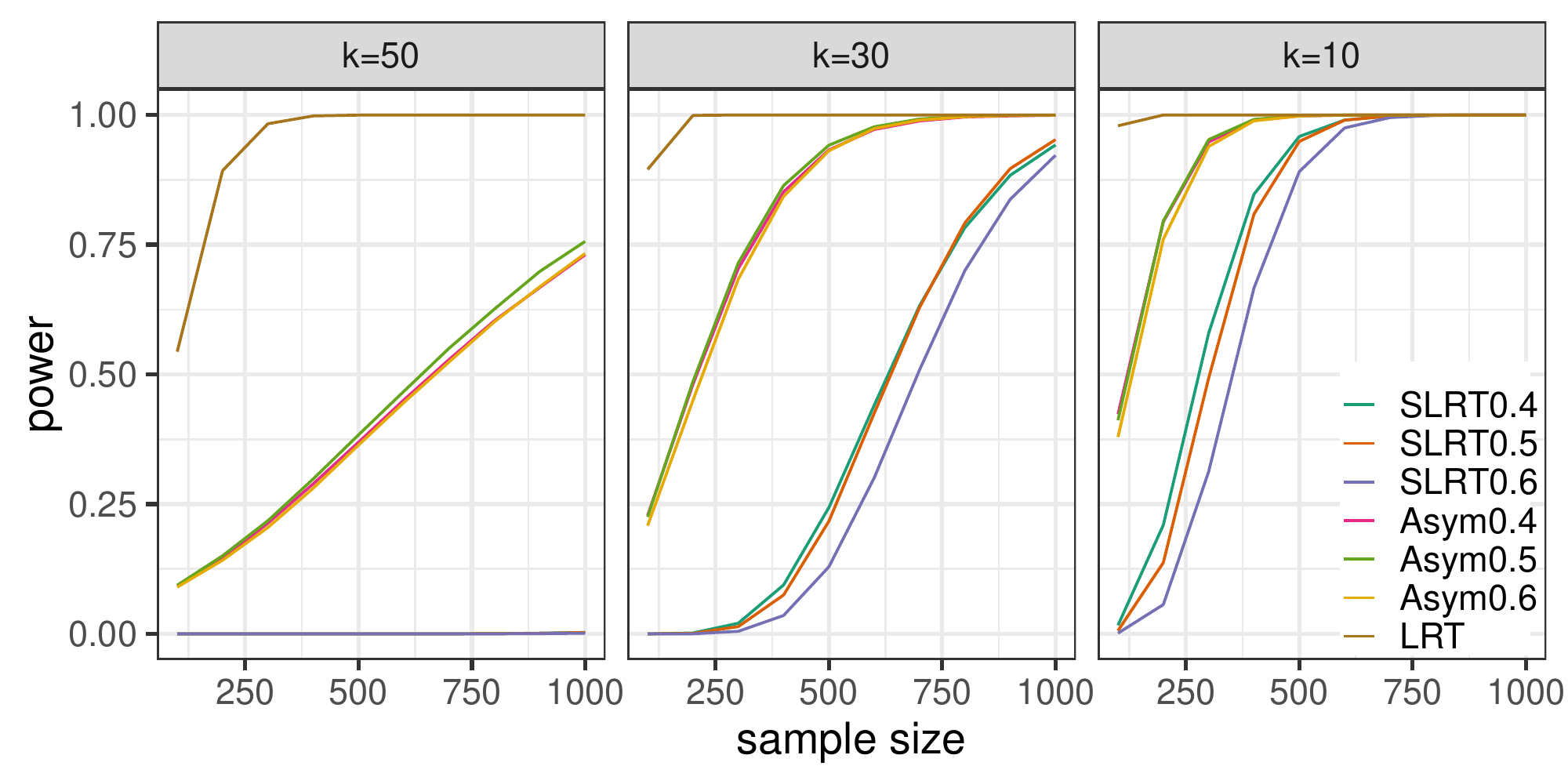}
		\caption{Power against sample size in Gaussian setting with $\theta=0.1$, $d=60$.}  \label{powern2}
    \end{figure}
    
    How much does using the universal threshold cost in terms of power asymptotically? In the following, we explore this question by comparing the power of the SLRT using the two different critical values, the  standard universal threshold (\texttt{SLRT}) and the quantile of the asymptotic distribution (\texttt{Asym}).  To this end we consider samples from a $d-$dimensional multivariate standard normal distribution with mean vector $(\theta,\dots,\theta)$  with $\theta=0.1$ and test the hypothesis that the first $d-k$ entries of $\theta$ equal zero.
	
    Figures \ref{powern} and \ref{powern2} display the (simulated) power of both variants as well as that of the classical LRT against the sample size. In this regular setting where the classical asymptotic distribution theory holds, the LRT outperforms the SLRT also when using the asymptotically correct quantiles. Furthermore, we see again that the power loss from using the universal threshold is larger in higher dimensions and higher dimensional null hypothesis settings. The simulations show that the choice of the splitting ratio plays an important role in the performance of the SLRT, especially in higher dimensional settings. In the following, we further examine the impact of the optimal choice of the splitting ratio in simulations.

	
	\subsection{Influence of the splitting ratio}\label{simulations2}

	\begin{figure}[t]
        \centering
		\includegraphics[width=0.75\linewidth]{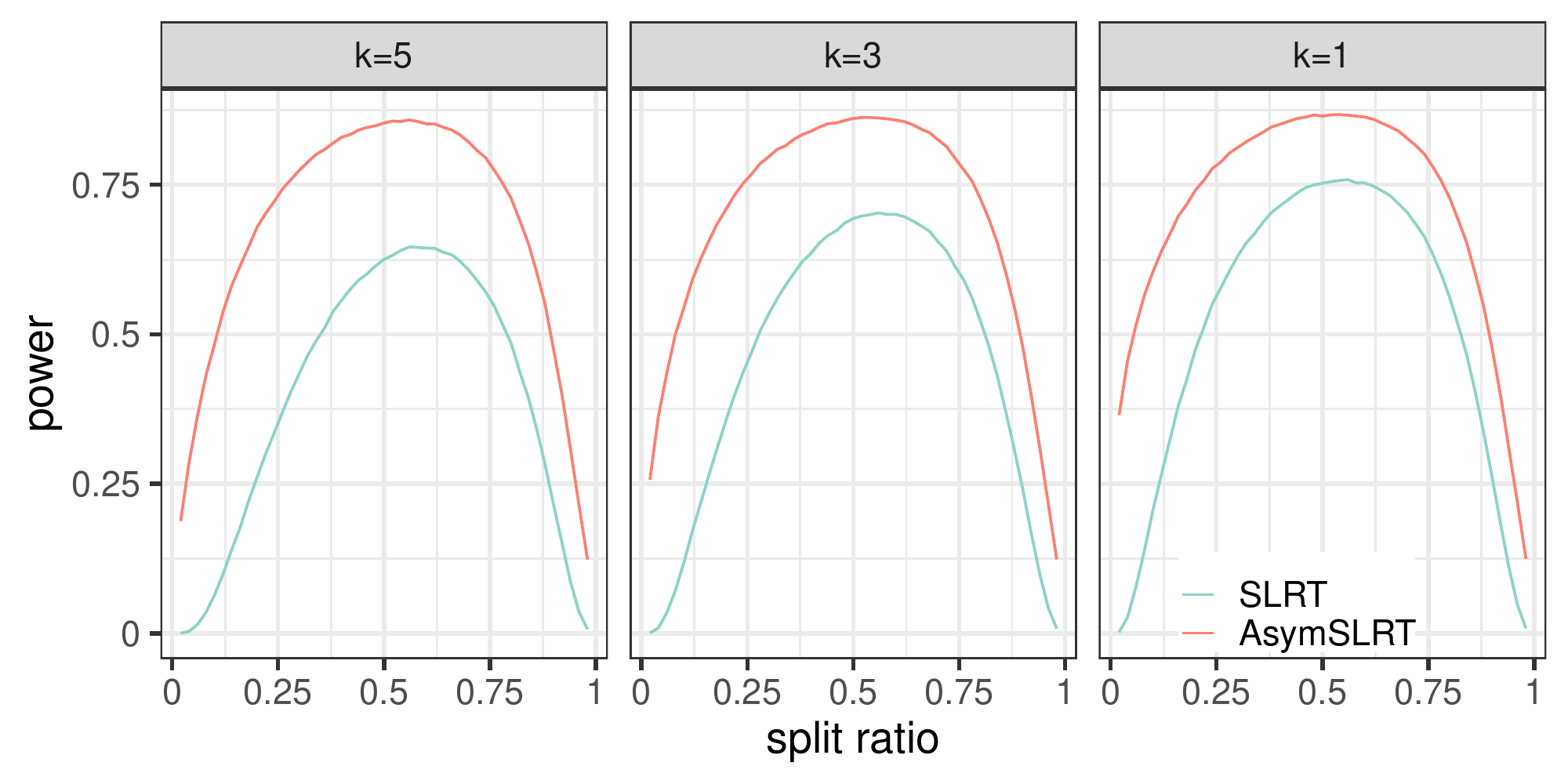}
			\caption{Power of SLRT against splitting ratio, $d=6$, $\delta=40$.} \label{powersplit}

        \centering
	    \includegraphics[width=0.75\linewidth]{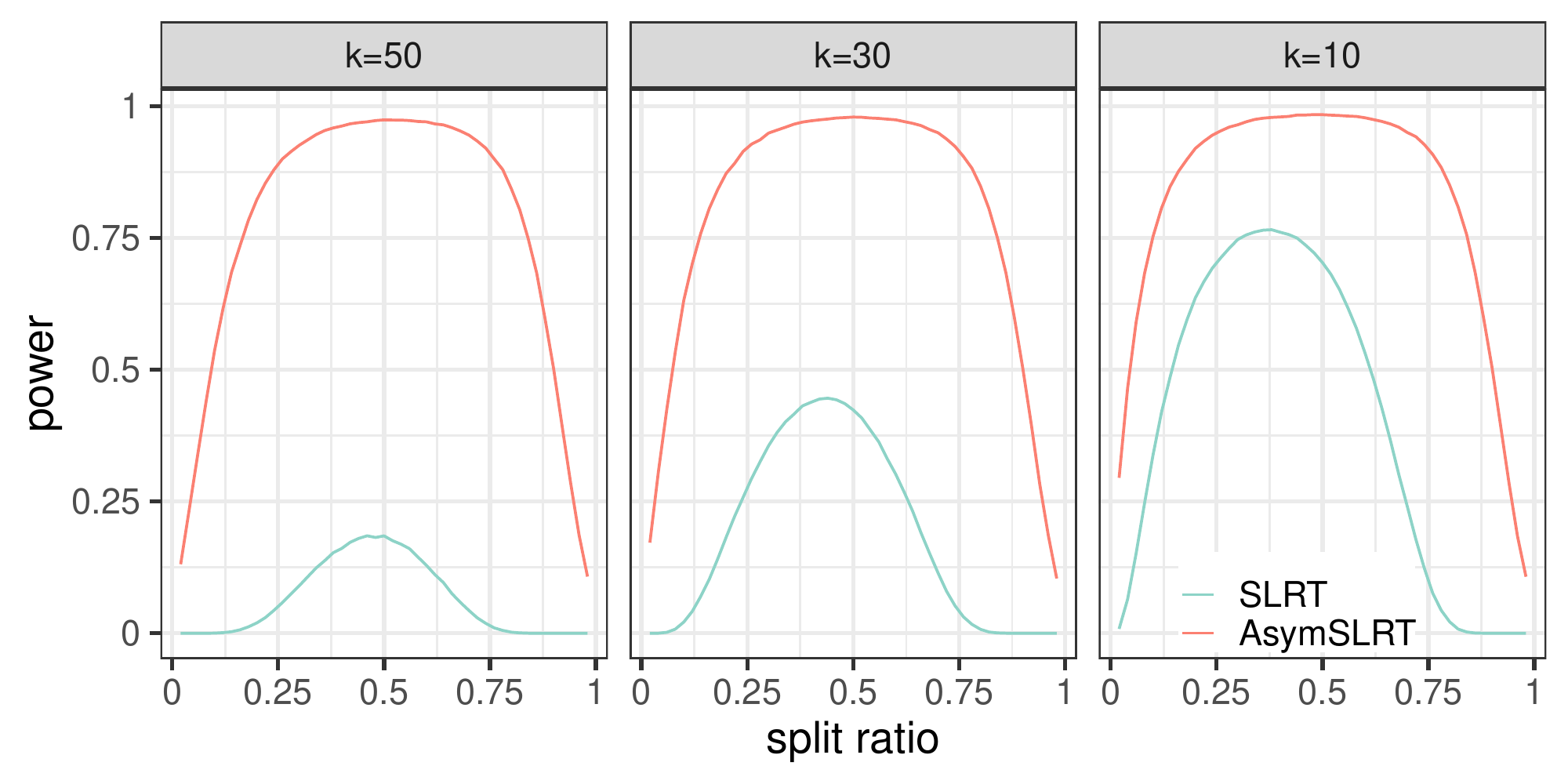}
		\caption{Power of SLRT against splitting ratio, $d=60$, $\delta=180$.} \label{powersplit2}
	\end{figure}

	In the following experiments, we analyze the influence of the splitting ratio on the asymptotic power of the SLRT.  To this end, we sample data from a noncentral split chi-square distribution and calculate the power for testing the hypothesis of a  zero noncentrality parameter $\delta$. Figures \ref{powersplit} and \ref{powersplit2} show the (simulated) power against the splitting ratio for the two different critical values, the universal threshold (\texttt{SLRT}) and the asymptotic quantile (\texttt{Asym}). We can see that in the lower dimensional setting a splitting ratio above $0.5$ performs best while in the higher dimensional setting a smaller splitting ratio below $0.5$ seems beneficial, especially for a lower dimensional null hypothesis.

	
    Figure \ref{poweropti} quantifies the improvement in power that can be achieved with our proposed (empirical) optimal splitting ratio (\texttt{emp.optim}) and the fast estimation routine (\texttt{est.optim}) that uses a normal approximation instead of extensive simulations to approximate the power of the SLRT. We compare the power for different noncentrality parameters $\delta \in \{100, 250\}$ plotted as 'dashed' and 'solid' lines respectively. Figure \ref{poweropti} displays that our fast estimation routine of the optimal split leads to valid approximations with a similarly good performance as the empirical optimal splitting ratio and further that there is a notable gain in power by using our new proposed optimal splitting ratios compared to the split \eqref{propsplit} by \citet{universal2}, especially in higher dimensions. 
    
    \begin{figure}[t]
    \centering
    \includegraphics[width=0.49\linewidth]{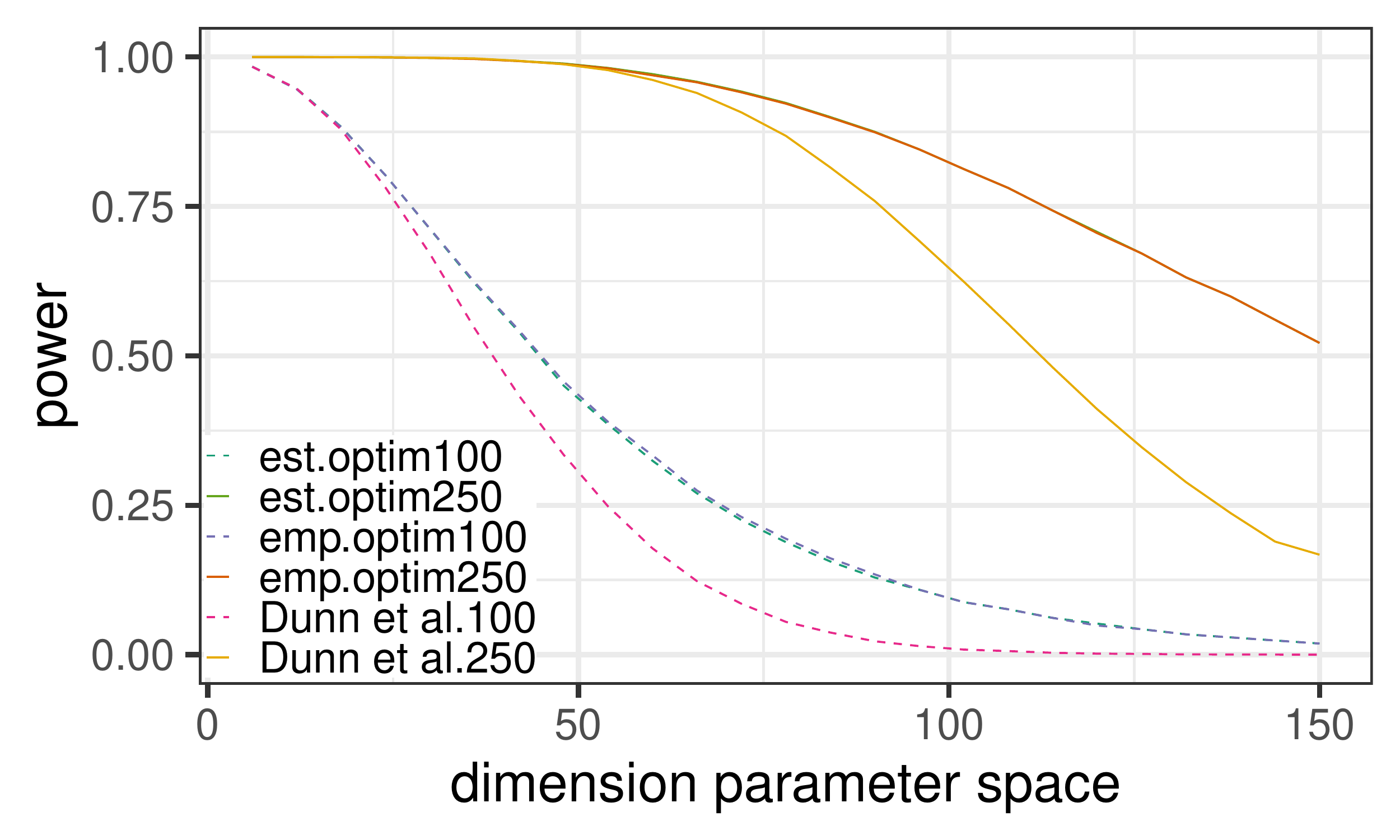}
	\includegraphics[width=0.49\linewidth]{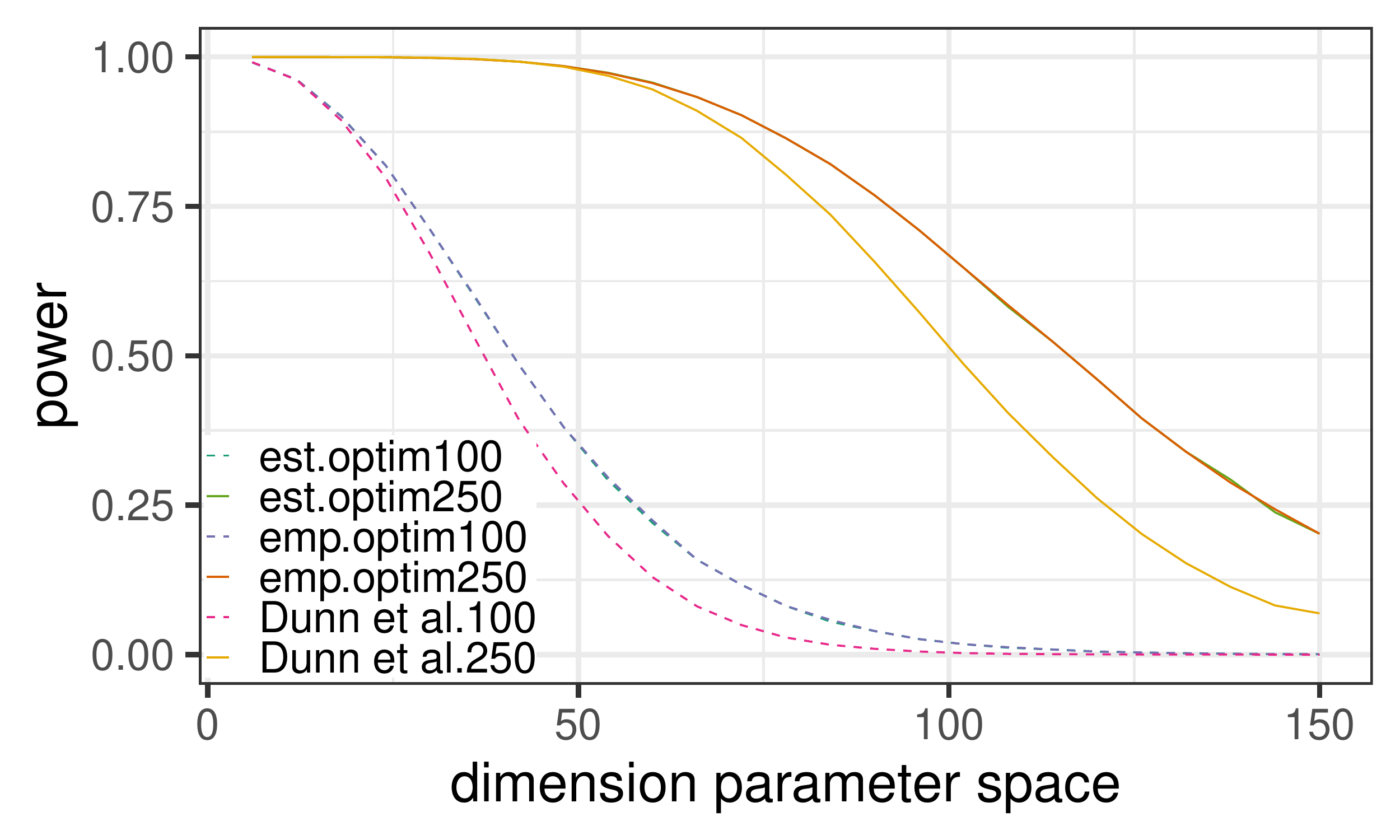}
	\caption{Power for fixed noncentrality parameter and $k=5$ (left); $k=\frac{1}{6} \,d$ (right).} \label{poweropti}
    \end{figure}
    
    This is even more apparent in Figure
    \ref{poweropticonstant}, where we calculated the power for two different regimes of increasing noncentrality parameter $\delta$. For each dimension of the parameter space, we chose the smallest $\delta$ such that the test with our new proposed optimal splitting ratio achieves a power of $0.8$ and $0.65$ respectively. While our methods, therefore, keep the power level, the split from~\eqref{propsplit} 
    leads to a rapid loss
    of power in higher dimensions.
    
    \begin{figure}[t]
	\centering
	\includegraphics[width=0.49\linewidth]{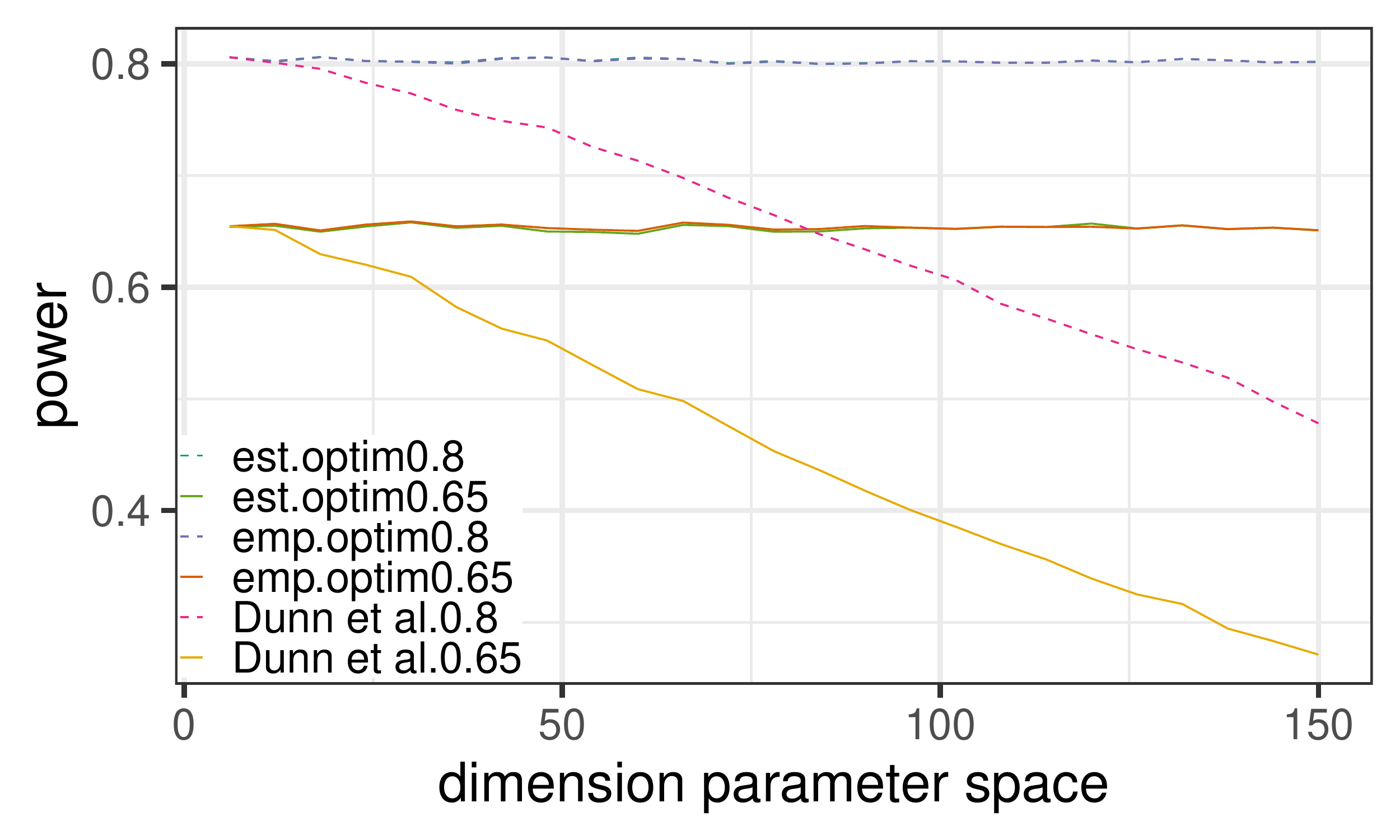}
    \includegraphics[width=0.49\linewidth]{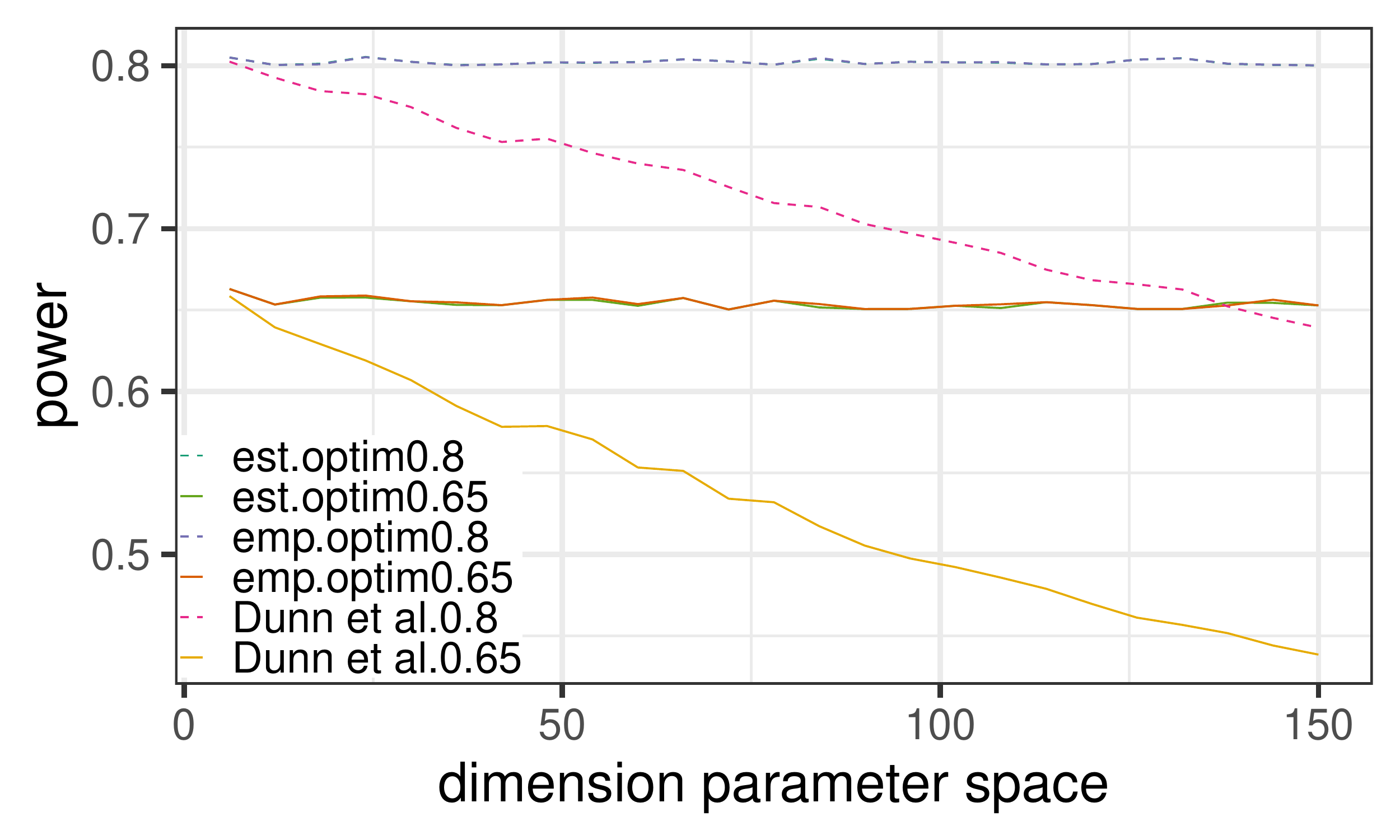}
	\caption{Power for increasing noncentrality parameter and $k=5$ (left); $k=\frac{1}{6} \,d$ (right)} \label{poweropticonstant}
	\end{figure}    
    
    \change{
    \subsection{Power of the SLRT in irregular setting}\label{section:powerirregular}
    Our proposed optimal splitting ratio is based on the limit distribution of the SLRT obtained under regularity conditions. In the following, we show that this optimal choice of the splitting ratio leads to an improvement in power even in irregular settings. To this end, we investigate the performance of the SLRT in the one-factor analysis setting with 12 observed variables. Assuming zero means, the one-factor model is given by the family of normal distributions $\mathcal{N}_{12}(0,\Sigma)$ with $\Sigma \in F_{12,1} := \{ \Omega + \Gamma \Gamma^T : \Omega \in \mathbb{R}^{12 \times 12}_{>0} \text{ diagonal, } \Gamma \in \mathbb{R}^{12} \}$. In the following experiment, we consider testing the one-factor model against the saturated alternative, that is, the entire cone of positive definite matrices $PD(12)$. Note that the hypothesis defines a $24$-dimensional subset of the $78$-dimensional parameter space, thus, our algorithm suggests using the optimal splitting ratio $0.41$ while \eqref{propsplit} suggests using the splitting ratio $0.51$. \citet{drton:2009} shows that in this one-factor analysis setting the hypothesis space has singularities at loadings $\Gamma$ with less than 3 nonzero values and thus, at those points, classic asymptotic distribution theory for likelihood ratio tests is not valid. }
    
    \begin{figure}[t]
	\centering
    \includegraphics[width=0.65\linewidth]{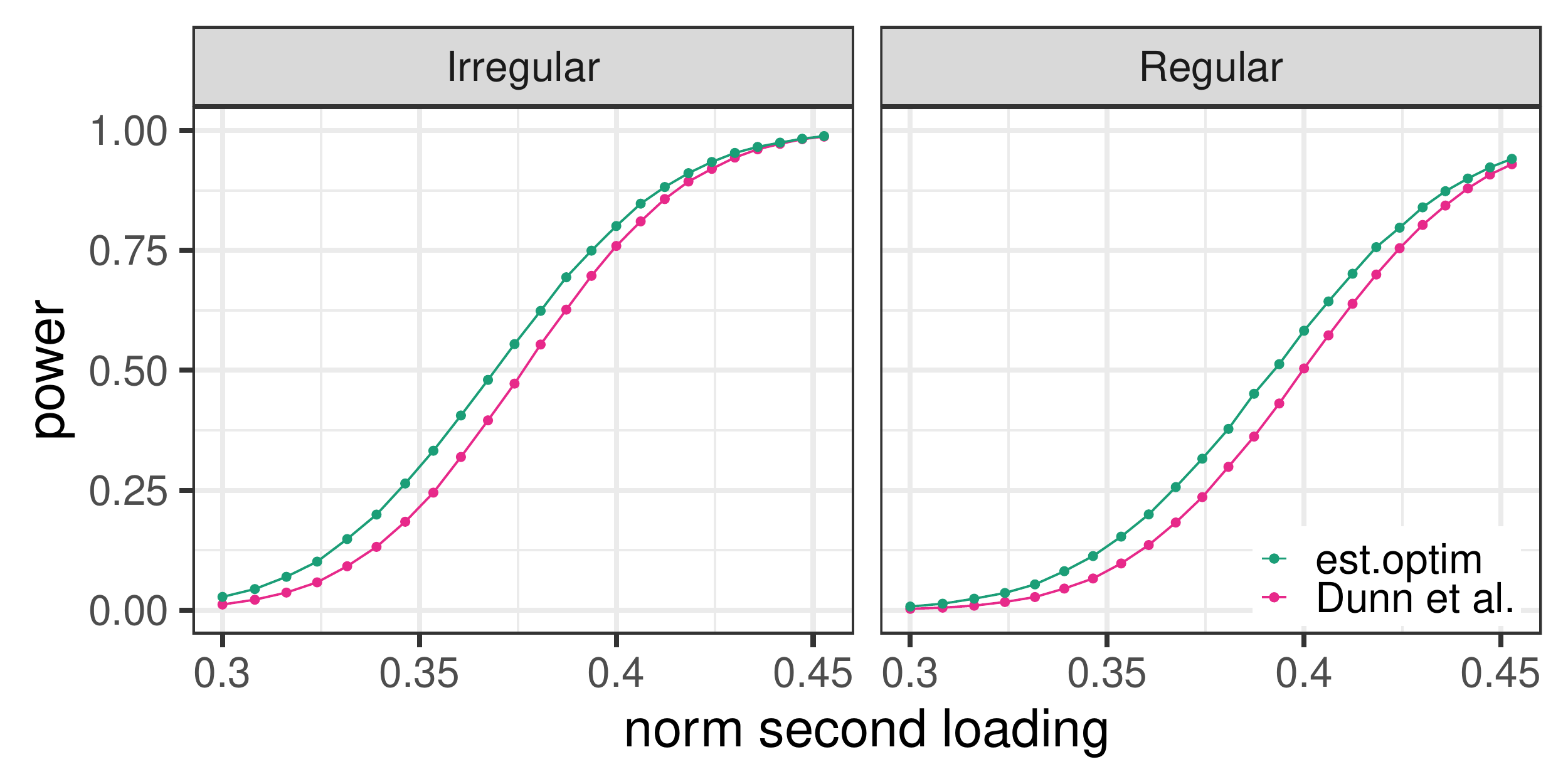}	\change{    
	\caption{Power of SLRT in one-factor analysis.} \label{factanal}}
	\end{figure}
    
    \change{
    Figure \ref{factanal} displays the power of the SLRT under alternatives around irregular and regular points using the two different splitting ratios. More specifically, we set $\Omega$ as the diagonal matrix with all diagonal entries equal $1/5$ and  $\Gamma=(5,5,0,...,0)$ for the irregular and $\Gamma=(5,5,5,0, ... ,0)$ for the regular setting, respectively. Then we sampled $n=2000$ data points from a two-factor alternative $\mathcal{N}_{12}(0,\Sigma)$ with $\Sigma=\Omega + \Gamma \Gamma^T + \Gamma_2 \Gamma_2^T$, where $\Gamma_2=(h/\sqrt{12}, ... ,h/\sqrt{12})$ with varying values of the norm $h$ of the second factor loading $\Gamma_2$. We observe, that the optimality of our proposed splitting ratio extends to this irregular setting and our splitting ratio outperforms the splitting ratio suggested by \citet{universal2} in all settings. }
    
\change{
\section{Extension to the cross-fit SLRT}
The SLRT and the universal inference framework starts with randomly splitting the available data. Thus, the value of the split likelihood ratio statistic still varies given a fixed data set, depending on the random split. It is natural to think about reducing this randomness by aggregating the results of different splits at the cost of performing more computations. Already in their initial work \citet{universal} propose the cross-fit SLRT, a variant of the SLRT where the test statistic is calculated twice with alternating roles of the two data sets.  Then, both results are averaged to obtain the cross-fit split likelihood ratio statistic $\tfrac{1}{2}(\Lambda_n+\Lambda_n^{swap})$, where $\Lambda_n^{swap}$ is defined by \eqref{split likelihood ratio stat} with the roles of $D_0$ and $D_1$ swapped.}

\change{
Similar to the proof of Theorem \ref{theorem1} we can derive the asymptotic distribution of the cross-fit split likelihood ratio statistic.
\begin{corollary}\label{corollarycross}
	Under assumptions (A1)-(A3), the cross-fit split likelihood ratio statistic satisfies under  $P_{\theta_n}$ with  $\theta_n=\theta_0 + h/\sqrt{n}$
    \begin{align*}
	\Lambda_n+\Lambda_n^{swap}
		\,\overset{\mathcal{D}}{\longrightarrow}\; &\Vert X + \sqrt{m_0}I(\theta_0)^{1/2}h-I(\theta_0)^{1/2} H_0\Vert^2 - \Vert X- \sqrt{\tfrac{m_0}{m_1}} Y \Vert^2 \\&+\Vert Y + \sqrt{m_1}I(\theta_0)^{1/2}h-I(\theta_0)^{1/2} H_0\Vert^2 - \Vert Y- \sqrt{\tfrac{m_1}{m_0}} X \Vert^2,
    \end{align*}
    where $X,Y \sim \change{\mathcal{N}_d(0,\mathrm{Id})}$ independent and $\Vert x - H_0 \Vert=\inf_{h \in H_0} \Vert x - h \Vert$.
\end{corollary}
Since this cross-fit variant of the SLRT only splits the data once and then uses the same data sets twice with alternating roles, the limit distribution is still generated by two independent random variables. Analogous to the SLRT, we can thus calculate the expectation and variance of the arising limit distribution in the smooth setting of Corollary \ref{cor1}, where the limiting hypothesis is a coordinate subspace, in the same way as the proof of Corollary \ref{cor2} by exploiting properties of quadratic forms.
\begin{corollary}\label{cor:moments}
    In the smooth setting of Corollary \ref{cor1} the expectation and variance of the limit distribution of the cross-fit split likelihood ratio test statistics is given by
    \begin{enumerate}
			\item 
			$
			\E[	\tfrac{1}{2}(\Lambda_{\infty}+\Lambda_{\infty}^{swap}) ]=	p-d-\frac{1}{2}d\big(\frac{m_0}{1-m_0}+\frac{1-m_0}{m_0}\big)+\frac{1}{2} \delta,
			$
			\item $
		    \mathrm{Var}[\tfrac{1}{2}(\Lambda_{\infty}+\Lambda_{\infty}^{swap})]=	(d-p)\big(1+ \frac{m_0}{1-m_0} + \frac{1-m_0}{m_0}\big) + d\big(2+ \frac{m_0}{1-m_0} + \frac{1-m_0}{m_0}\big)$\\ \hspace*{3.1cm} $\quad + \frac{1}{2}d \Big(\frac{m_0^2}{(1-m_0)^2}+\frac{(1-m_0)^2}{m_0^2}\Big) + \delta,
			$
			\end{enumerate}
   where $\delta=[I(\theta_0)^{1/2}h]_{[p]}^T[I(\theta_0)^{1/2}h]_{[p]}.$
\end{corollary}}
\change{
Furthermore, analogously to Section \ref{subsec:optimal}, we can employ the limit distribution to determine the optimal splitting ratio and obtain the intuitive result that for the cross-fit SLRT an even split of $m_0=0.5$ is optimal in all dimensions. 
\begin{remark} \label{remark:crossfit}
    For an equal splitting ratio $m_0=0.5$ the cross-fit split likelihood ratio statistic has the same expectation but a lower variance than the split likelihood ratio statistic in the limit. Thus, in situations where the method has power, the cross-fit SLRT further improves upon the SLRT. 
\end{remark}
}

\change{The cross-fit SLRT employs equal weights $w_0=0.5$ to combine both test statistics and obtain the cross-fit split likelihood ratio test statistic, that is $w_0\Lambda_n+(1-w_0)\Lambda_n^{swap}$. Using equal weights is intuitive for an even splitting ratio $m_0=0.5$, since both test statistics are expected to perform similarly. However, considering the results from the previous section, for different splitting ratios, it might be beneficial to vary the weights and emphasize the better-performing test statistic. Furthermore, the idea of the cross-fit SLRT of swapping the roles of the two data sets allows us to derive properties of the limit distribution and is conceptually simple. Nevertheless, in view of our previous analysis, subsampling provides a promising alternative. Instead of using the same split data set with swapped roles to calculate the second test statistic, we repeat the process of randomly splitting the available data. Such a method has a similar computational burden as the cross-fit SLRT, however, we can make use of the optimal splitting ratio for both test statistics. Note that this subsampling procedure can be extended to multiple repeats and thus, further decrease the randomness of the splits at the cost of computation time. }

 \begin{figure}[t]
	\centering
    \includegraphics[width=0.65\linewidth]{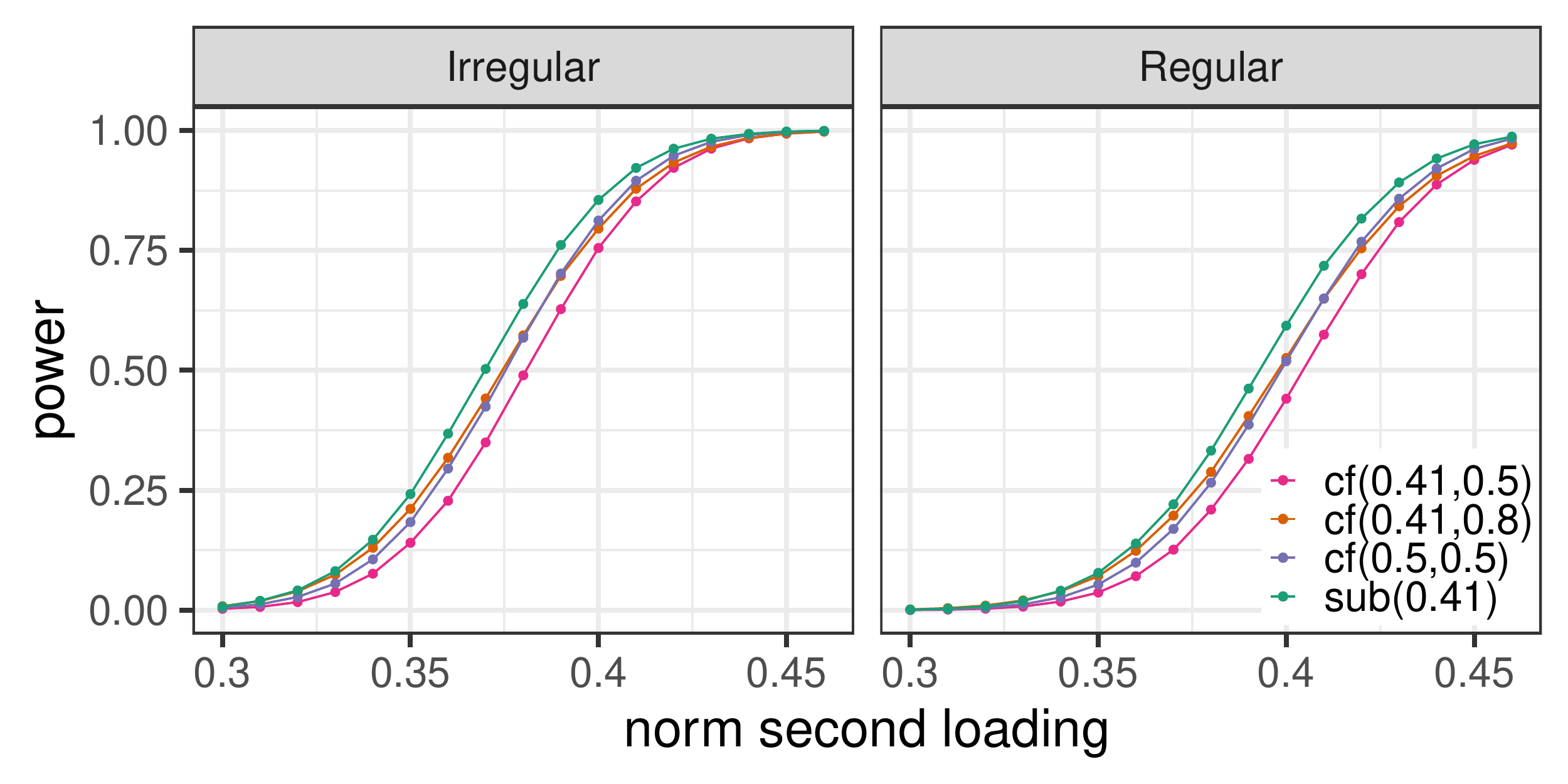}	\change{    
	\caption{Power of cross-fit($m_0$, $w_0$) and subsampling SLRT in one-factor analysis.} \label{factanalcross}}
\end{figure}

\change{
Figure {\ref{factanalcross}} compares the performance of the different extensions of the SLRT in the one-factor analysis setting introduced in Section {\ref{section:powerirregular}}. We display the power of the cross-fit SLRT with different splitting ratios $m_0$ and different weights $w_0$ as well as the subsampling alternative with our proposed optimal splitting ratio $m_0=0.41$. As previously mentioned, under the assumption of fixed, equal weights $w_0=0.5$, the best-performing splitting ratio for the cross-fit SLRT is an even split $m_0=0.5$. However, for an uneven splitting ratio, we can emphasize the better-performing test statistic by adjusting the weights to achieve similar performance. Furthermore, Figure {\ref{factanalcross}} shows that using our proposed optimal splitting ratio for both test statistics in a subsampling procedure outperforms the competition in all settings.}

\section{Conclusion}
\label{discussion}

The split likelihood ratio test (SLRT) is a flexible tool that provides valid level $\alpha$ tests in finite samples even when classical regularity conditions are not satisfied. The underlying universal approach of splitting the data allows one to conduct rather simple analyses even in complicated  settings. In general, this flexibility leads to a rather conservative method and, thus, it is of interest to carefully choose the splitting ratio in order to mitigate possible loss of power.  

In order to provide new insights about the performance of the SLRT we studied its asymptotic behavior in the setting of smooth hypotheses.  Our study gives rise to a new class of distributions, noncentral split chi-square distributions, that appear as limiting distributions of the SLRT.  The split chi-square distribution depends on the dimensions of both null and alternative hypotheses and not only the difference of the dimensions.  Naturally, it also depends on the chosen  data splitting ratio.   Using the new class of distributions, we analyzed the power of the SLRT in extensive simulations, and we proposed a new routine for calculating the optimal splitting ratio for the SLRT that significantly boosts power, especially in higher dimensions. 

\begin{acks}[Acknowledgments]
This project has received funding from the European Research Council (ERC) under the European Union's Horizon 2020 research and innovation programme (grant agreement No.~83818). Further, this work has been funded by the German Federal Ministry of Education and Research and the Bavarian State Ministry for Science and the Arts. The authors of this work take full responsibility for its content.

\end{acks}


\bibliographystyle{imsart-nameyear} 
\bibliography{universalinference.bib}       

\begin{thebibliography}{8}

\bibitem[\protect\citeauthoryear{Drton}{2009}]{drton:2009}
\begin{barticle}[author]
\bauthor{\bsnm{Drton},~\bfnm{Mathias}\binits{M.}}
(\byear{2009}).
\btitle{Likelihood ratio tests and singularities}.
\bjournal{Ann. Statist.}
\bvolume{37}
\bpages{979--1012}.
\bmrnumber{2502658}
\end{barticle}
\endbibitem

\bibitem[\protect\citeauthoryear{Dunn et~al.}{2021}]{universal2}
\begin{bmisc}[author]
\bauthor{\bsnm{Dunn},~\bfnm{Robin}\binits{R.}},
  \bauthor{\bsnm{Ramdas},~\bfnm{Aaditya}\binits{A.}},
  \bauthor{\bsnm{Balakrishnan},~\bfnm{Sivaraman}\binits{S.}} \AND
  \bauthor{\bsnm{Wasserman},~\bfnm{Larry}\binits{L.}}
(\byear{2021}).
\btitle{Gaussian Universal Likelihood Ratio Testing}.
\end{bmisc}
\endbibitem

\bibitem[\protect\citeauthoryear{Hartigan}{1985}]{hartigan:1985}
\begin{binproceedings}[author]
\bauthor{\bsnm{Hartigan},~\bfnm{J.~A.}\binits{J.~A.}}
(\byear{1985}).
\btitle{A failure of likelihood asymptotics for normal mixtures}.
In \bbooktitle{Proceedings of the {B}erkeley conference in honor of {J}erzy
  {N}eyman and {J}ack {K}iefer, {V}ol. {II} ({B}erkeley, {C}alif., 1983)}.
\bseries{Wadsworth Statist./Probab. Ser.}
\bpages{807--810}.
\bpublisher{Wadsworth, Belmont, CA}.
\bmrnumber{822066}
\end{binproceedings}
\endbibitem

\bibitem[\protect\citeauthoryear{Kendall and Stuart}{1969}]{kendall}
\begin{bbook}[author]
\bauthor{\bsnm{Kendall},~\bfnm{Maurice~G.}\binits{M.~G.}} \AND
  \bauthor{\bsnm{Stuart},~\bfnm{Alan}\binits{A.}}
(\byear{1969}).
\btitle{The advanced theory of statistics. Vol. 1: Distribution theory},
\bedition{Third} ed.
\bpublisher{Hafner Publishing Co., New York}.
\bmrnumber{0246399}
\end{bbook}
\endbibitem

\bibitem[\protect\citeauthoryear{Strieder et~al.}{2021}]{uai}
\begin{binproceedings}[author]
\bauthor{\bsnm{Strieder},~\bfnm{David}\binits{D.}},
  \bauthor{\bsnm{Freidling},~\bfnm{Tobias}\binits{T.}},
  \bauthor{\bsnm{Haffner},~\bfnm{Stefan}\binits{S.}} \AND
  \bauthor{\bsnm{Drton},~\bfnm{Mathias}\binits{M.}}
(\byear{2021}).
\btitle{Confidence in causal discovery with linear causal models}.
In \bbooktitle{Proceedings of the Thirty-Seventh Conference on Uncertainty in
  Artificial Intelligence}
(\beditor{\bfnm{Cassio}\binits{C.}~\bparticle{de} \bsnm{Campos}} \AND
  \beditor{\bfnm{Marloes~H.}\binits{M.~H.}~\bsnm{Maathuis}}, eds.).
\bseries{Proceedings of Machine Learning Research}
\bvolume{161}
\bpages{1217--1226}.
\bpublisher{PMLR}.
\end{binproceedings}
\endbibitem

\bibitem[\protect\citeauthoryear{Tse and Davison}{}]{tsedavison}
\begin{barticle}[author]
\bauthor{\bsnm{Tse},~\bfnm{Timmy}\binits{T.}} \AND
  \bauthor{\bsnm{Davison},~\bfnm{Anthony~C.}\binits{A.~C.}}
\btitle{A Note on Universal Inference}.
\bjournal{Stat}
\bpages{e501}.
\end{barticle}
\endbibitem

\bibitem[\protect\citeauthoryear{van~der Vaart}{1998}]{vdv}
\begin{bbook}[author]
\bauthor{\bparticle{van~der} \bsnm{Vaart},~\bfnm{A.~W.}\binits{A.~W.}}
(\byear{1998}).
\btitle{Asymptotic statistics}.
\bseries{Cambridge Series in Statistical and Probabilistic Mathematics}
\bvolume{3}.
\bpublisher{Cambridge University Press, Cambridge}.
\bmrnumber{1652247}
\end{bbook}
\endbibitem

\bibitem[\protect\citeauthoryear{Wasserman, Ramdas and
  Balakrishnan}{2020}]{universal}
\begin{barticle}[author]
\bauthor{\bsnm{Wasserman},~\bfnm{Larry}\binits{L.}},
  \bauthor{\bsnm{Ramdas},~\bfnm{Aaditya}\binits{A.}} \AND
  \bauthor{\bsnm{Balakrishnan},~\bfnm{Sivaraman}\binits{S.}}
(\byear{2020}).
\btitle{Universal inference}.
\bjournal{Proceedings of the National Academy of Sciences}
\bvolume{117}
\bpages{16880--16890}.
\end{barticle}
\endbibitem

\end{thebibliography}


\end{document}